\documentclass[11pt]{amsart}
\usepackage{amsfonts, amstext, amsmath, amsthm, amscd, amssymb}
\usepackage{epsfig, graphics, psfrag, enumerate}
\usepackage{color}
\usepackage{verbatim}
\let\oldmarginpar\marginpar
\renewcommand\marginpar[1]{\oldmarginpar[\raggedleft\footnotesize #1]%
{\raggedright\footnotesize #1}}

\renewcommand{\setminus}{{\smallsetminus}}

\newcommand{\bdy}{{\partial}}

\newcommand{\abs}[1]{{\left\vert #1 \right\vert}}

\theoremstyle{definition}
\theoremstyle{plain}
\newtheorem{theorem}{Theorem}[section]
\newtheorem{corollary}[theorem]{Corollary}
\newtheorem{lemma}[theorem]{Lemma}
\newtheorem{prop}[theorem]{Proposition}

\newcommand{\lcomb}{\ell_c}

\newtheorem*{namedtheorem}{\theoremname}
\newcommand{\theoremname}{testing}
\newenvironment{named}[1]{\renewcommand{\theoremname}{#1}\begin{namedtheorem}}{\end{namedtheorem}}

\theoremstyle{definition}
\newtheorem{define}[theorem]{Definition}

\newtheorem{example}[theorem]{Example}

\newtheorem{algorithm}[theorem]{Algorithm}

\begin{document}
\title[Crosscap numbers and the Jones polynomial ]{Crosscap numbers and the Jones polynomial}
\author[E. Kalfagianni]{Efstratia Kalfagianni}
\author[C. Lee]{Christine Ruey Shan Lee}

\address[]{Department of Mathematics, Michigan State University, East
Lansing, MI, 48824}

\email[]{kalfagia@math.msu.edu}

\address[]{Department of Mathematics, The University of Texas at Austin, Austin, TX 78712}

\email[]{leechr29@msu.edu}

\thanks{{This research was supported in part by NSF grants DMS--1105843 and DMS-1404754 and by  the
Erwin Schr\"odinger International Institute for Mathematical Physics.}}
\thanks{Lee is  supported by an NSF Research Postdoctoral Fellowship (DMS-1502860). }

\thanks{ \today}

\begin{abstract} We give sharp two-sided linear bounds of the crosscap number (non-orientable genus)  of  alternating links
in terms of their Jones polynomial. Our estimates are often exact and we use them to calculate the crosscap numbers for several infinite families
of alternating links and for
several alternating knots with up to twelve crossings.  We also discuss generalizations of our results for classes of non-alternating 
links. \end{abstract}

\maketitle

\section{Introduction}

The goal of this paper is to give
 two-sided linear bounds of the \emph{crosscap number} (i.e. the non-orientable genus)  of an alternating link
in terms of coefficients of the Jones polynomial of the link.  We show that both of these bounds are  sharp
and often they give the exact value of the crosscap number. As an application we calculate the crosscap number of several
infinite families of alternating links.
We also check that our bounds give the crosscap numbers of 283
alternating knots with up to twelve crossings that were previously unknown.
Finally, we generalize our results to classes of non-alternating links.

To state our results,
for a link $K \subset S^3$, let
$$J_K(t)= \alpha_K t^n+ \beta _Kt^{n-1}+ \ldots + \beta' _Kt^{s+1}+ \alpha'_K
t^s$$  denote the Jones polynomial of $K$, so that $n$ and $s$
denote the highest and lowest power in $t$. Set 
$$T_K:= \abs{\beta_K} + \abs{\beta'_K},$$
\noindent where $\beta_K$ and
$\beta'_K$ denote 
 the
second and penultimate coefficients of $J_K(t)$, respectively.  Also
let $s_K= n-s$, denote the degree span of $J_K(t)$.

The \emph{crosscap number} of a non-orientable surface with $k$ boundary components is defined to be
$2-\chi(S)-k$. The \emph{crosscap number} of a link $K$ is the minimum crosscap number over all non-orientable surfaces 
spanning $K$.

\begin{theorem}\label{thm:cupjones}
 Let $K$ be a non-split, prime alternating link  with $k$-components and  with crosscap number $C(K)$.
Suppose that $K$ is not a $(2,p)$ torus link.
We have

$$ \left\lceil \frac{T_K}{3}\right\rceil  +2-k\; \leq \; C(K) \; \leq \; T_K+2-k,
  $$

\noindent where $T_K$ is  as above, and
$\lceil \cdot \rceil$ is the ceiling function that rounds up to the nearest larger integer.
Furthermore, both  bounds are sharp.
\end{theorem}

By a result of Menasco \cite{menasco:incompress} a link with a connected, prime alternating diagram that is not 
the standard diagram of a $(2, p)$ torus link is non-split, prime and non-torus link. Hence the hypotheses of Theorem 
\ref{thm:cupjones} are easily checked from alternating diagrams.

In the 80s  Kauffman \cite{Kaufjones} and Murasugi \cite{murasugitait}
showed   that  the degree span of the Jones polynomial  determines the crossing number of alternating links. More recently,
Futer, Kalfagianni and Purcell showed
 that
coefficients 
of the colored Jones polynomials 
contain information about
 incompressible surfaces in the link complement and have  strong relations to 
 geometric structures  and in particular
to hyperbolic geometry  \cite {fkp:qsf, fkp,  fkp:survey}. For instance,  certain coefficients
of the polynomials coarsely determine the volume of large classes of hyperbolic links \cite{fkp:filling, fkp:farey},
including hyperbolic alternating links as shown by Dasbach and Lin  \cite{dasbach-lin:volumish}. 
In fact, the Volume Conjecture \cite{murakamisurvey} predicts that certain asymptotics of the colored Jones polynomials  determine the
volume of all hyperbolic links.
Furthermore, it has been conjectured that the degrees of  the colored Jones polynomials  determine
slopes of incompressible surfaces in the link complement
\cite{garoufalidis:jones-slopes}. 
Theorem \ref{thm:cupjones} gives a new relation of the Jones polynomial to a fundamental topological knot invariant
and prompts several interesting
questions about the topological content  of
 quantum link invariants. See discussion in Section \ref{secfive}.

Upper bounds of the knot crosscap number have been previously discussed in the literature.
 Clark \cite{clark} observed that for  \emph{any}  knot $K$,  we have $C(K)\leq 2g(K)+1$, where
 $g(K)$ is the orientable genus of $K$.
 Murakami and Yasuhara \cite{upper} showed that 
 $$C(K)\leq \left\lfloor{ \frac{c(K)}{2}}\,\right\rfloor,$$
 where $c(K)$ is the crossing number of $K$ and 
 $\lfloor \cdot \rfloor$ is the floor function that rounds up to the nearest smaller integer. Both of these bounds upper bounds are known to be sharp.
 
 To the best of our knowledge, before the results of this paper, the only known lower bound for $C(K)$, of an alternating link $K$,
was that $C(K)>1$, unless $K$ is a $(2, p)$ torus link.\footnote{
A lower bound for the 4-dimensional  crosscap number of \emph{all} knots, and thus for the 3-dimensional crosscap number,
was given by Batson \cite{batson}. For alternating knots, however, this bound is non-positive.}

Combining Theorem  \ref{thm:cupjones} with the results of \cite{upper}
and \cite{Kaufjones}, we have the following.
 
 \begin{theorem}\label{thm:cupjonesknots}
 Let $K$ be an alternating, non-torus knot  with crosscap number $C(K)$ and let $T_K$ be as above.
We have

$$ \left\lceil \frac{ T_K}{3}\right\rceil  + 1\; \leq \; C(K) \; \leq \;  {\rm {min}}{ \left\{ T_K + 1, \
\left\lfloor{ \frac{s_K}{2}}\,\right\rfloor \right\}}
$$
where $s_K$  denotes the degree span of $J_K(t)$.
Furthermore, both bounds are sharp.
 \end{theorem}

 Crowell \cite{crowell} and Murasugi \cite{murasugi} have independently shown that the orientable genus of an alternating knot is equal to half the degree span of the
Alexander polynomial of the knot. Theorem \ref{thm:cupjonesknots} can be thought of as the non-orientable analogue of  this classical result. We will have more to say about this in Section \ref{secfive}.

The orientable link genus has been well studied, and a general algorithm for  calculation, using normal surface theory,  is known  \cite{genus, genusalg}. For low crossing number knots, effective computations can also be made from genus bounds coming from invariants such as the Alexander polynomial and the Heegaard Floer homology \cite{knotinfo}.
Crosscap numbers, however, are   harder to compute. Although the crosscap numbers of several special  families of knots are known
(\cite{pretzel, torus, 2-bridge}), no effective general method  of calculation is known.
Some progress in this direction
was made by Burton and Ozlen \cite{burton} using normal surface theory and integer programming. However, at the time this writing, there is no-known normal surface
algorithm to determine crosscap numbers.
In particular, for the majority of prime knots up to twelve crossings the crosscap numbers are listed as unknown
in  Knotinfo \cite{knotinfo}.

For alternating links, a method to compute crosscap numbers was given by Adams and Kindred \cite{Adamsstate}.  They showed
that to compute  the crosscap number of an alternating link it is enough to find the minimal crosscap number realized by \emph{state surfaces} corresponding
to alternating link diagrams.   Note that for a link with $n$-crossings, there  correspond $2^{n}$ state surfaces that, \emph{a priori} one has to search and select one with minimal crosscap number. The algorithm of \cite{Adamsstate} cuts down significantly this number, but the number of surfaces that one has
to deal with still grows fast as the number of the ``non-bigon" regions of the alternating link does.
The advantage of Theorem \ref{thm:cupjones}  is that it provides estimates that are easy to calculate from any alternating diagram and, as mentioned above, these estimates often compute the exact crosscap number.  Indeed, the quantity $T_K$ is particularly easy to calculate
from any alternating knot diagram  as each of $\beta_K, \beta_K'$ can be calculated from the checkerboard 
graphs of the diagram \cite{dfkls:determinant}.  
We've checked that our lower bound  improves the lower bound given in Knotinfo \cite{knotinfo} 
for 1472 prime alternating knots for which the crosscap number is listed as unknown
and for 283 of these knots our bounds determine the exact value of the crosscap number. See Section \ref{calculations}.

In the proof of Theorem \ref{thm:cupjones} we make use of the results of \cite{Adamsstate}.
We show that given an alternating diagram $D(K)$, a spanning surface from which the crosscap number of $K$ is easily determined,
 can be taken to lie in the complement of an \emph{augmented link}
obtained from $D(K)$. Then, we use a construction essentially due to Adams \cite{Adamsaug}, presented by  Agol, D. Thurston \cite[Appendix]{lackenby:volume-alt},
 ideas due to Casson and Lackenby \cite{lack-surg} and a result of
Futer and Purcell \cite{futer-purcell}, 
to prove  Theorem \ref{thm:cupjones}. In particular, we make use of the fact that augmented link complements admit angled polyhedral decompositions with several nice combinatorial and geometric features.
We employ normal surface theory  and a combinatorial version of a Gauss-Bonnet theorem to estimate the Euler characteristic of surfaces that realize crosscap numbers of alternating links.
Using these techniques we show that for prime alternating links, the crosscap number is bounded in terms of the \emph{twist number}
of any prime, twist-reduced, alternating projection (for the definitions see Section \ref{sectwo}). 

In particular, for knots we have the following:

\begin{theorem}\label{thm:chi-etimate2}
Let $K \subset S^3$ be a knot with a prime, 
twist-reduced alternating diagram  $D(K)$. Suppose that $D(K)$ has $t \geq 2$ twist
regions and let  $C(K)$ denote the crosscap number of $K$.
We have

$$1+ \left\lceil \frac{ t}{3}\right\rceil \; \leq \; C(K) \; \leq \;  {\rm {min}}{ \left\{ t+1, \ 
\left\lfloor{ \frac{c}{2}}\,\right\rfloor \right\}}
$$
where $c$ denotes the number of crossings of $D$.
Furthermore, both bounds are sharp.
\end{theorem}

Having related $C(K)$ to twist numbers of alternating link projections, Theorems \ref{thm:cupjones} and \ref{thm:cupjonesknots} follow by  \cite{dasbach-lin:volumish}.

The overall technique used to prove Theorem \ref{thm:chi-etimate2} goes beyond the class of alternating links and allow us to generalize Theorem \ref{thm:cupjones}  for 
large classes of non-alternating links. In  Theorem \ref{thm:cupjonesadequate}  we provide two-sided linear bounds of $C(K)$ in terms of $T_K$ for \emph{adequate} links
that admit a link diagram with at least six crossings in each twist region (see Section \ref{secfive} for definitions and terminology).

We've made an effort to make the paper self-contained:
 In Section \ref{sectwo} we define augmented links and state definitions and results
from \cite{lack-surg, futer-purcell} that we need in this paper in the particular forms that we need them.
In Section \ref{secthree} first we define state surfaces and we recall the results of \cite{Adamsstate} that we use.
Then we prove  Theorems \ref{thm:cupjones} and \ref{thm:cupjonesknots}.
In Section \ref{calculations} we  calculate the crosscap numbers of infinite families of alternating knots for which the lower bound is sharp
as well as for several knots
up to 12 crossings. Finally, in Section \ref{secfive}  we discuss generalizations of our results outside the class of alternating links and we state some questions that
arise from this work.

\section{Augmented links and estimates with normal surfaces}\label{sectwo}
Consider  a link diagram $D(K)$ as a 4--valent graph
in the plane, with over--under crossing information associated to each
vertex.  A bigon region is a region of the graph bounded by only two
edges.  A \emph{twist region} of a diagram consists of maximal
collections of bigon regions arranged end to end. We will assume that the crossings in each twist region occur in an alternating fashion.
A single crossing
adjacent to no bigons is also a twist region. 

\begin{figure}[]
\includegraphics[scale=.5]{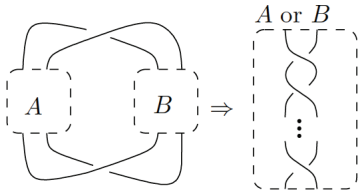}
\caption{Twist reduced: $A$ or  $B$ must be a string of bigons.}
\label{fig:treduced}
\end{figure} 

\begin{define} A link diagram $D(K) $ is \emph{prime} if any simple
closed curve which meets two edges of the diagram transversely
bounds a region of the diagram with no crossings. 

The diagram $D(K)$ is called \emph{twist reduced},
if any simple closed curve that meets the diagram transversely in four
edges, with two points of intersection adjacent to one crossing and
the other two adjacent to another crossing, bounds a (possibly empty) collection of bigons arranged end
to end between the crossings.  See Figure \ref{fig:treduced}, borrowed from \cite{fkp:filling}.
\end{define}

\subsection{An angled polyhedral decomposition.} For a link $K\subset S^3$, let $\eta(K)$ denote a regular neighborhood of $K$ and let
$E(K)$ denote the exterior of $K$; that is $E(K)=\overline{S^3\setminus \eta(K)}$.

Let $D(K)$ be a prime,
twist--reduced diagram of a link $K$. For every twist region of
$D(K)$, we add an extra link component, called a \emph{crossing
  circle}, that wraps around the two strands of the twist region.  A choice of crossing links for each twist region of $D(K)$ gives a new link $J$, called an \emph{augmented link}. 
The link $L$ obtained by removing all full twists from the twist regions of $J$ is called a \emph{fully augmented link}.
The exteriors $E(J)$, $E(L)$ are homeomorphic 3-manifolds and $E(K)$ can be
expressed as a Dehn filling of $E(L)\cong E(J)$. See
Figure \ref{fig:augment}, borrowed from \cite{fkp:filling}.

\begin{figure}[]
\begin{center}
\psfrag{K}{$K$}
\psfrag{J}{$J$}
\psfrag{L}{$L$}

\includegraphics[scale=0.30]{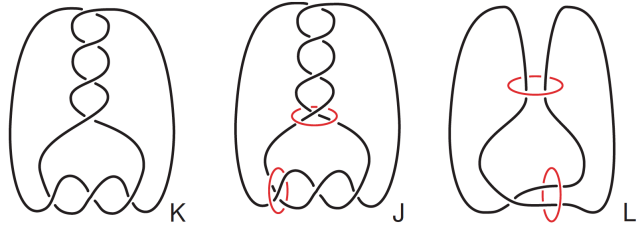}
\caption{A link  $K$, an augmented link $J$ and a fully augmented link  $L$.}
\label{fig:augment}
\end{center}
\end{figure}

For twist regions consisting
of a single crossing the addition of a crossing circle can be done in two ways. 
See Figure ~\ref{fig:twoways}. 
\begin{figure}[ht]
\includegraphics[scale=.16]{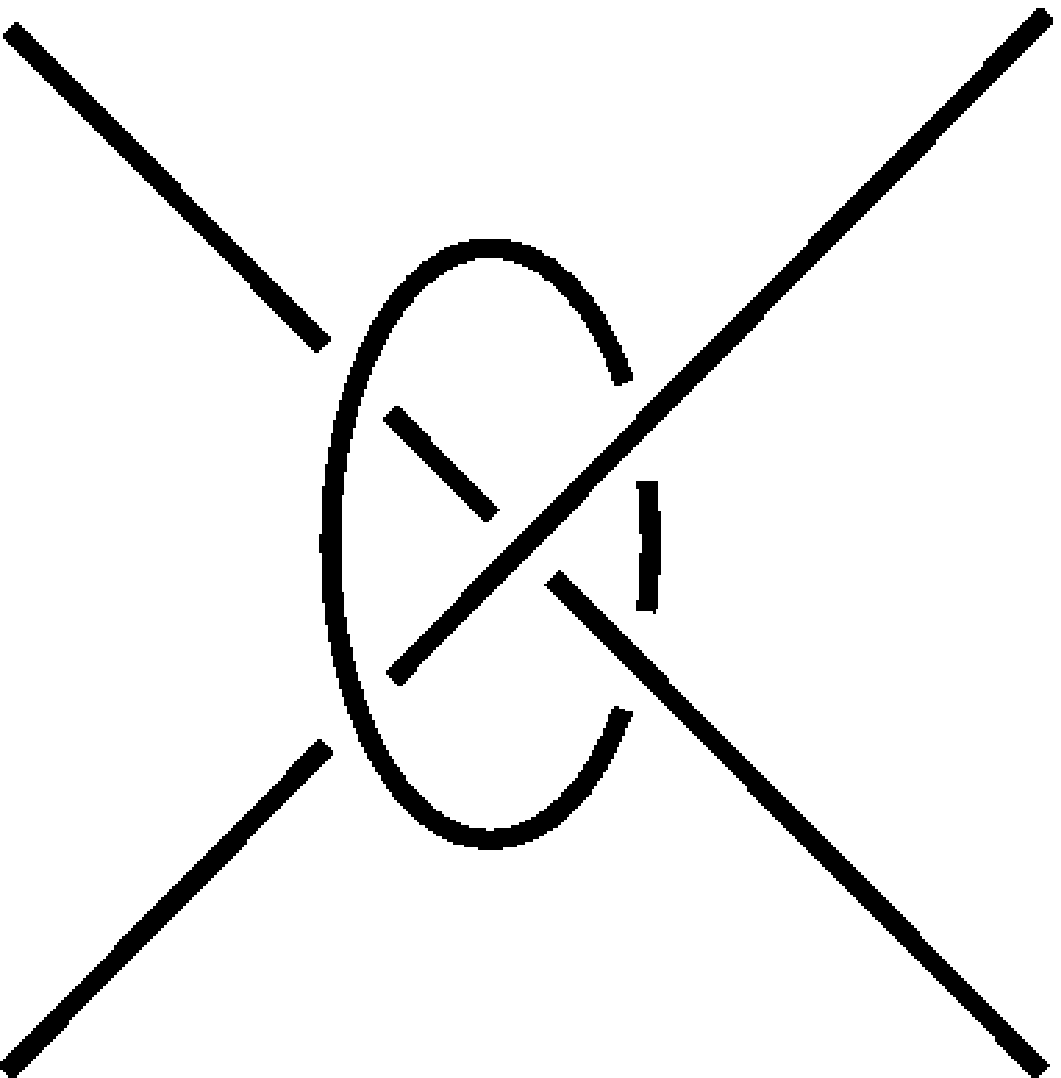}
\hspace{3cm} 
\includegraphics[scale=.16]{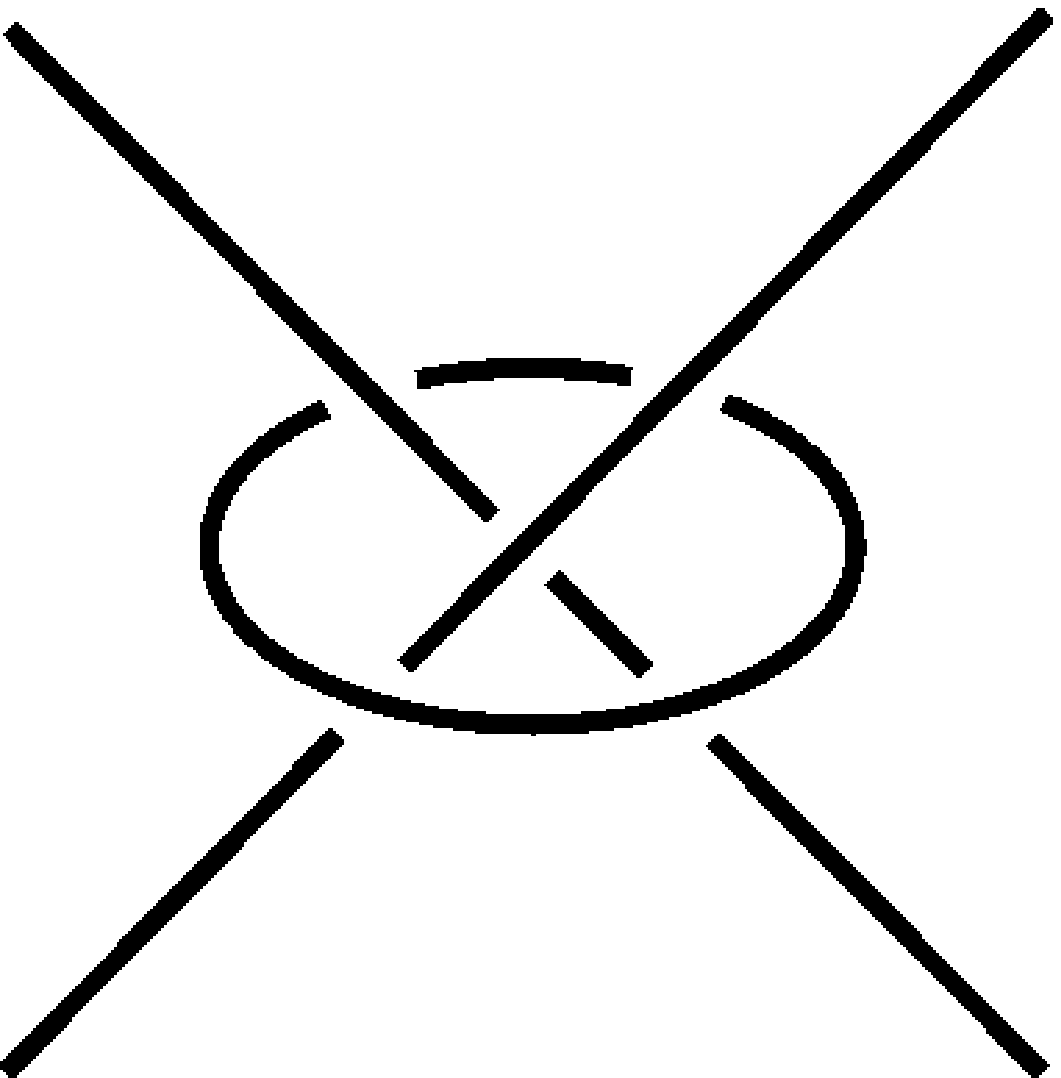}
\caption{ Two ways of augmenting a twist region with a single crossing.}
\label{fig:twoways}
\end{figure} 

The geometry of augmented links, first studied by Adams \cite{Adamsaug}, is well understood. Below, we  will summarize some results and properties we need; for more details see  Purcell's
 expository article  \cite{purcellsurvey} and references therein.
We will consider the  non-crossing circle components of $L$  flat on the projection plane and the crossing circles
bound crossing disks vertical to the projection plane. By work of Adams,
as explained by Agol and D. Thurston in the appendix of \cite{lackenby:volume-alt}, $E(L)$
 has  an \emph{angled polyhedral decomposition} with nice geometric and combinatorial properties. 
In this paper we are particularly interested in the combinatorial structure of this decomposition:
 one can define the \emph{combinatorial area} of surfaces in $E(L)$ that are in \emph{normal form}
with respect to the polyhedral decomposition and \emph{combinatorial length} of simple closed curves that lie  the boundary
$\partial{E(L)}$. The general setting was discussed by Lackenby in  \cite{lack-surg} building on ideas and constructions first introduced by Casson.
These ideas where applied in \cite{lack-surg} to prove the so called ``6-Theorem" in Dehn filling and were
 applied by
 Futer and Purcell   \cite{futer-purcell} to study the geometry of ``highly twisted" links and their Dehn fillings.

Start with a prime,  twist reduced diagram of a link $K$
and let $L$ be a fully augmented link obtained from it.
The manifold $E(L)$  
can be subdivided  into two identical,  convex
ideal polyhedra $P_1$ and $P_2$ in the hyperbolic 3-space. 
The ideal vertices of each polyhedron $P\in \{ P_1, P_2\}$
are 4-valent and they correspond to crossing circles or to arcs of $K\setminus {\mathcal D}$,
where $\mathcal D$ is the union of all the crossing disks.
After truncating all the  vertices of $P$ we have:
\begin{itemize}
\item Dihedral angles  at each edge of $P$ are $\pi/2$.

\item There are two types of edges  of $\partial P$: these  that are created by truncation called \emph{ boundary edges} (edges of $P\cap \partial E(L)$)
and the ones that
come from edges existing before truncation, called \emph {interior edges}. The interior edges come from intersections of the crossing discs with the projection plane.

\item There are two types of faces on $\partial P$: these  that are created by  truncation called \emph{ boundary faces} (faces of $P\cap \partial E(L)$),
and the ones that come from faces existing before truncation, called  \emph{interior faces}.
\item Each boundary face is a rectangle that meets four interior edges at the vertices of the rectangle.
The boundary faces of $P_1, P_2$ subdivide  $\partial E(J)$
into rectangles.

\item  The interior faces of $P$ can be colored with two colors (shaded and white) in a  checkerboard fashion so that at each rectangular boundary face of $P$
opposite side interior faces have the same color. The shaded faces correspond to crossing disks while the white faces correspond to regions of the projection
plane. See \cite[Figure 15]{lackenby:volume-alt} and Figure \ref{panel}.

\item  The properties  of the decomposition can in particular be used to prove  that $E(L)$ is hyperbolic.

\end{itemize}

The complement of all faces (interior and boundary)  on the projection plane is a graph $\Gamma\subset \partial P$ with vertices of valence at least three
and edges  the edges of $P$.

A polyhedral $P$ with the above properties  is called a \emph{rectangular-cusped} polyhedron.

\subsection{ Normal surfaces and combinatorial area.} In this paper we are interested in surfaces with boundary that are properly embedded in $E(L)$ and are in \emph{normal
form} with respect to the above polyhedral decomposition. We recall the following definition.

\begin{define} \label{normal} A properly embedded 
surface $(F, \partial F)\subset (E(L), \partial E(L))$
is said to be in \emph{normal form} with respect to the polyhedra decomposition,  if for any $P\in \{ P_1, P_2\}$
we have the following:
\begin{enumerate}
\item $F \cap P$ consists of properly embedded disks  $(D, \partial D)\subset (P, \partial P)$.
\item $F\cap \partial P$ is a collection  of simple closed curves  none of which lies entirely in a single face of $P$.

\item $F$ intersects faces of $P$ in a collection of properly embedded arcs none of which
passes through vertices of $\Gamma$. Furthermore, none of these arcs runs from an edge of $F$ to itself or from an interior edge
to an adjacent boundary edge.

\item A component  of  $F \cap P$ can intersect each boundary face in at most one arc.

\end{enumerate}
 \vskip 0.05in

\noindent The components of  $F \cap P$ are called \emph{normal disks}. 
An example of  three normal disks in a
truncated polyhedron is shown in Figure \ref{disks}, borrowed from \cite{futer-purcell}. Note that $P$ shown there is not a rectangular-cusped polyhedron as 
not all the boundary faces are rectangles.
\end{define}

\begin{figure}[ht]
\begin{center}

\includegraphics[scale=0.28]{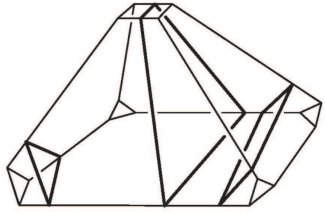}
\caption{Normal disks in a truncated polyhedron.}
\label{disks}
\end{center}

\end{figure}

We recall the definition of the  \emph{combinatorial} area of a surface in normal form.

\begin{define} \label{area} Let  $(D, \partial D)\subset (P, \partial P)$ be a normal disk in
a polyhedron  $P\in \{ P_1, P_2\}$. Suppose that  $D$ crosses  $m$   interior edges of $P$. 
The combinatorial area of $D$, denoted by  $a(D)$, is defined by
$$a(D)= \frac{m\pi}{2}+ \pi |D\cap \partial E(L)|- 2\pi,$$
\noindent where $|D\cap \partial E(L)|$ denotes the number of arcs of $\partial D$ running on boundary faces of $P$.
For an embedded 
surface $(F, \partial F)\subset (E(L), \partial E(L))$, in normal form, the combinatorial area $a(F)$ is defined by summing over all the normal disks of $F$ in the polyhedra $P_1, P_2$.
\end{define}

Next we recall the following form the  combinatorial version of the Gauss-Bonnet; it is special case of   \cite[Proposition 4.3]{lack-surg}.

\begin{prop} \label{gaussBonnet} Let $(F, \partial F)\subset (E(L), \partial E(L)))$ be a surface in normal form of Euler characteristic
$\chi(F)$. We have
$$a(F) = -2\pi \chi(F).$$

\end{prop}
\vskip 0.03in

To continue, with $P$ as above,  consider a normal disk    $(D, \partial D)\subset (P, \partial P)$ such that  $\partial D$  intersects at least one boundary face of $P$.
Given an arc of  $\gamma \subset \partial D$ on a boundary face of $P$, define the \emph{combinatorial length}
of $\gamma$ with respect to $D$ by
$$l(\gamma, D)=\frac{a(D)}{|D\cap \partial E(L)|}.$$

For a simple closed curve  $\gamma \subset \partial (E(L)$, that is a boundary component of a surface
$(F, \partial F)\subset (E(L), \partial E(L))$, let $H\subset F$ be the union of normal disks in $F$
whose intersections with the boundary faces of $P_1, P_2$ give $\gamma$. Thus $\gamma$ is the union of arcs each properly embedded in a boundary face of the polyhedra.
Now define the combinatorial length

$$l_c(\gamma):= l(\gamma, H)\, =\, \sum_{i}^{} l(\gamma_i, D),$$ 

\noindent where the sum is taken over all normal disks in $H$ and all normal arcs on boundary faces.

The  quantity $l_c(\gamma)$, defined above,  depends on $F$. To obtain a well defined notion of \emph{combinatorial length} 
one needs to consider the {\it {infimum}} over all normal surfaces $F$ and collections $H$ with $\partial F=\gamma$.
In fact, one can define the combinatorial length of \emph{any } simple closed curve on $\partial E(L)$  \cite{lack-surg, futer-purcell}.
This is the definition of $l_c$ used by the authors in \cite{futer-purcell}.
We will not repeat these definitions here as we don't need them.   However, the  combinatorial length estimates  of simple closed curves 
obtained in \cite{futer-purcell} also hold for $l(\gamma, H)$.
We need the following lemma that follows immediately from the above definitions.

\begin{lemma}{\rm \cite[Lemme 4.13]{futer-purcell}}\label{length-area}
Let $(F, \partial F)\subset (E(L), \partial E(L)))$ be an embedded surface in normal form with respect to the polyhedral
decomposition and let $\gamma_1, \ldots, \gamma_k$ denote the
components of $\bdy F$.
Then, 
$$a(F) \geq \sum_{j=1}^{k} \lcomb(\gamma_j) \, .$$
\end{lemma}
 
Finally we need the following.

\begin{lemma} \label{gen-estimate} Suppose that $E(L)$
is an augmented link obtained from a prime, twist-reduced link diagram $D(K)$. Let  $(F, \partial F)\subset (E(L), \partial E(L)))$ 
be an embedded  normal surface and let $\gamma$ be a component of $\partial F$ that is homologically non-trivial
on a component 
$T\subset \bdy E(L)$.
 If $T$ comes from a crossing circle  of $L$, let $m$
denote the number of crossings in the corresponding twist region of $D(K)$.
If $T$ comes from a
component $K_j$ of $K$, let $m$ be the number of twist regions visited
by $K_j$, counted with multiplicity.  We have
$$\lcomb(\gamma) \geq \frac{m\pi }{3}.$$
\end{lemma}
\begin{proof} It is proven in the proof of
\cite[Corollary 5.12]{futer-purcell} using \cite[Proposition 5.3]{futer-purcell}.
\end{proof}

\subsection{ Genus estimates for spanning  surfaces}  
 Let $(S, \partial S)\subset (E(K), \partial E(K))$  be a spanning surface of $K$.
That is a surface where the components of $\partial S$ are the components of $K$ and it contains no closed components. In fact, often, we will use
(some times implicitly) the following convenient  characterization of spanning surfaces.

\begin{lemma}\label{spanning} A properly embedded surface $(S, \partial S)\subset ( E(K), \partial E(K))$, without closed components,
is a spanning surface of $K$ iff for each component of $\partial E(K)$, the total geometric intersection number of the  boundary curves $\partial S$ with
the corresponding meridian is $1$.
\end{lemma}
\begin{proof} If $S$ is a spanning surface of $K$, then clearly
the desired conclusion holds. Conversely, suppose that  each component of  $\partial S$ has geometric  intersection number  $1$ with the meridian on the component of $\partial E(L)$ it lies.
Then $\partial S$ must have exactly one component on the corresponding component of $ \partial E(K)$.
For, if we had  more that one curves on some component of $ \partial E(K)$  then these curves will be parallel creating more intersections of
$\partial S$ with the corresponding meridian that would contribute to the geometric intersection number. Thus each component of 
$\partial S$ is a longitude of $\partial E(L)$.
\end{proof}

Let $J$ be an augmented link obtained from a diagram of a $K$ and let $L$ be the corresponding fully augmented
link. A spanning surface $(S, \partial S)\subset (E(K), \partial E(K))$  gives  a
 punctured surface $(F, \partial F)\subset ( E(J), \partial E(J))$.
Now $F$ gives a properly embedded surface   $(F', \partial F')\subset ( E(L), \partial E(L))$.
Since we are only interested in the Euler characteristic of the surface we will often choose to work with $F'$
instead of $F$. In fact, abusing the setting, we will identify $F$ with $F'$ and say we can view $F$ as 
a surface 
in the exterior of the fully augmented link $E(L)\cong E(J)$.
For the next theorem we will also assume
that  the surface $F$ can be isotopied into normal form with respect to the polyhedral decomposition of $E(L)$.

\begin{theorem} \label{chi-etimate}
  Let $K \subset S^3$ be a link of $k$ components with a prime,
  twist-reduced diagram  $D(K)$. 
Suppose that $D(K)$ has $t \geq 2$ twist
  regions and  let $\tau$ denote the smallest number of crossings
corresponding to a twist region of $D(K)$. Let $S$ be a spanning surface
of $K$ and
let $L$ be an augmented link, obtained from $K$, such that $S$ intersects
$n$ crossing circles of $L$.
Suppose  that the corresponding punctured surface $F\subset E(L)$ can be isotopied to be  normal with respect to the polyhedral decomposition $P_1, P_2$.
Then we have 

  $$-\chi(S)\geq \left\lceil  \frac{t}{3} \, +\,  \frac{n \tau}{6}\, -\, n
  \right\rceil.$$
 \end{theorem}

\begin{proof}
  The boundary $\partial F$ consists of curves $\gamma_1,
  \ldots, \gamma_k$ one for each torus component of $\partial E(L)$ that comes from a component of $K$ and curves $\gamma_{k+1}, \ldots, \gamma_{k+n}$ on the components coming from crossing circles.
By assumption $F$  can be isotopied into  normal form in the polyhedra $P_1$ and $P_2$; so we can compute
  its combinatorial area. By applying Proposition \ref{gaussBonnet}  and Lemma \ref{length-area} we have

\begin{eqnarray*}
-2\pi \cdot \chi(S)
&=& -2\pi \cdot \chi(F)-2\pi n =a(F)-2\pi n \\
&\geq& \sum_{i=1}^{k} \ell(\gamma_i) \, + \, \sum_{i=1}^{n} \ell(\gamma_{k+i}) -2\pi n . \\
\end{eqnarray*}

By Lemma
  \ref{gen-estimate}, the total length of the curves  $\gamma_1, \ldots, \gamma_k$
  is at least $2t\pi/3$, because $K$ passes through each twist region
  twice.  By the same lemma the total length of the curves $\gamma_{k+1}, \ldots, \gamma_{k+n}$
is at least $n\tau\pi/3$. Thus from the last equation we obtain

\begin{eqnarray*}
-2\pi \cdot \chi(S)
&\geq&  \frac{2t\pi}{3}  \, + \, \frac{n\tau\pi}{3}-2\pi n\\
&=& 2\pi ( \frac{t}{3} \, +\, \frac{n\tau }{6}-n). \\
\end{eqnarray*}
Since the Euler characteristic is an integer, the conclusion follows.
\end{proof} 

Recall that for a non-orientable surface $S$, with  $k$ boundary components,
 the crosscap number is  defined to be $C(S)=2-\chi(S)-k$.
We have the following result that should
be compared with \cite[Theorem 1.5]{futer-purcell}.
\begin{corollary}  \label{twisted} Let the notation and setting be as in Theorem \ref{chi-etimate}. Suppose moreover that
$D(K)$ has at least six crossings in each twist region and that $S$ is non-orientable. Then we have
$$C(S)\geq \left\lceil  \frac{t}{3}\right\rceil \, +\,  2\, -\, k .$$
\end{corollary}
\begin{proof} Let $\tau$ denote the smallest number of crossings
corresponding to a twist region of $D(K)$.  By assumption, $\tau\geq 6$.
Thus Theorem \ref{chi-etimate} gives 
\begin{eqnarray*}
C(S)
&=&
-\chi(S)\, +2\, -\, k
\geq 
\left\lceil  \frac{t}{3} \, +\,  \frac{n \tau}{6}\, -\, n \right\rceil \, +2\, -\, k \\
&\geq &
\left\lceil  \frac{t}{3} \, +\,  n \, -\, n \right\rceil \, +2\, -\, k = 
\left\lceil  \frac{t}{3}  \right\rceil \, +2\, -\, k.\\
\end{eqnarray*}

\end{proof}


\section{Alternating links}\label{secthree}
Adams and Kindred   \cite{Adamsstate}  gave an  algorithm, starting with an alternating link projection, to
construct  spanning surfaces of maximal Euler characteristic
among all the spanning surfaces of the  link. 
Our goal in this section is to show that we can take such a spanning surface to lie in
the complement of an appropriate augmented link obtained from the link projection. In the next section we will use the techniques of
Section \ref{sectwo} to estimate the crosscap number of a prime alternating link  in terms of  the twist number and the crossing number of any prime, twist-reduced  alternating 
diagram of the link.

\subsection{State surfaces and a minimum overall genus algorithm} Given a crossing on a link diagram $D(K)$ there are two ways to resolve
it. 
A Kauffman state $\sigma$ on $D(K)$ is a choice of one of these two resolutions at each crossing of $D(K)$.  
Given a state $\sigma$ of $D(K)$ we obtain a  spanning surface $S_{\sigma}$ of $K$, as follows:
The result of applying $\sigma$
to $D(K)$
is a collection $v_{\sigma}(D)$ of non-intersecting
circles in the plane,  called \emph{state circles}. We record the crossing resolutions along $\sigma$ by embedded segments connecting the state circles.
 Each circle of $v_{\sigma}(D)$ bounds a disk in $S^3$. These disks may be nested on the projection plane but can be made disjointly embedded
in $S^3$ by pushing  their interiors at different heights below the projection plane.
For each arc recording the resolution of a crossing of $D(K)$ in $\sigma$, we connect the pair of neighboring disks by a half-twisted band. The result is a surface $S_{\sigma} \subset S^3$ whose boundary is $K$.  See Figure \ref{resolve}. 
\begin{figure}[ht]
\includegraphics[scale=.6]{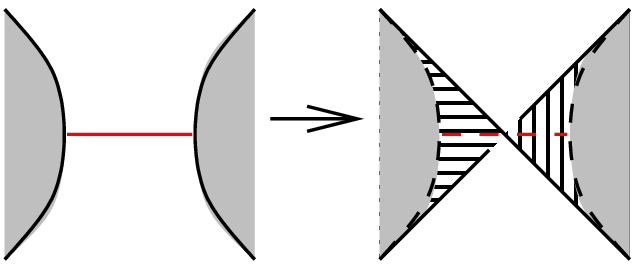}
\hspace{2cm}
\includegraphics[scale=.6]{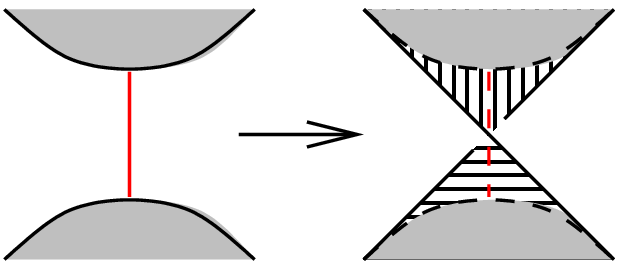}

\caption{The two resolutions of a crossing, the arcs recording them  and their contribution to state surfaces.} 
\label{resolve}
\end{figure}

Note that there might be several ways to make the disks bounded by the circles $v_{\sigma}(D)$ disjoint and the resulting surfaces may not necessarily be isotopic
in $S^3$. Nevertheless, this point is not important for the results of this paper, or for those of \cite{Adamsstate}
as the resulting surfaces will have the same topology (i.e. Euler characteristic and orientability). The geometry of the state surfaces, $S_{\sigma}$ with the particular construction described above, was studied by Futer, Kalfagianni  and Purcell  \cite{fkp, fkp:qsf}. 

In \cite{Adamsstate} the authors show that state surfaces of alternating diagrams can be used to determine the crosscap numbers of alternating links. In particular, starting with an alternating diagram $D(K)$,  they gave an algorithm to obtain a state surface with maximal Euler characteristic among all the state surfaces of $D(K)$. Then they showed that this Euler characteristic is the maximum over all spanning surfaces of $K$, state or non-state. 
To review this algorithm,  recall that a   diagram  $D(K)$ may be viewed as  a 4-valent graph on   $S^2$ with over and under information at each vertex. A complementary region of this graph on $S^2$ is an $m$-gon if its boundary consists of $m$ vertices or edges. We will refer to a $2$-gon as a bigon and a $3$-gon as a triangle. We need the following elementary lemma. 

\begin{lemma} \label{triangle} Suppose that an alternating diagram $D(K)$ contains no 1-gon or bigon.  Then it must contain at least one triangle.
\end{lemma}

\begin{proof}
Consider the 4-valent graph on $S^2$ defined by $D(K)$ and let $V$, $E$, $F$ 
denote the number of its vertices, edges and complimentary regions, respectively.
 Then,
$V - E + F = 2$.
 We have $E = 4V / 2 = 2V$. Hence $V - 2V + F = 2$, which implies that $F > V$. Let $m$ be the smallest positive integer for which the 4-valent graph contains an $m$-gon. If it contains no 
 1-gon, bigon or triangle, then $m > 3$. This implies that $F < 4V / 4 = V$ since each face must have at least four  distinct vertices in its boundary and each vertex can only be on the boundary of at most 4 distinct faces. This is a contradiction. Therefore, $m \not> 3$ and if $m > 2$, then $m =3$.
\end{proof}

Observe that the Euler characteristic of a state surface,
corresponding to a state $\sigma$ that results to $v_{\sigma}$ state circles, is   $\chi(S_{\sigma})=v_{\sigma}-c$, where $c$ is the number of crossings on $D(K)$.
Thus to maximize $\chi(S_{\sigma})$ we must maximize the number of state circles $v_{\sigma}$.
Now we review  the algorithm of  \cite{Adamsstate}:  

\vskip 0.03in

\begin{algorithm} \label{AKalgor} {\rm Let $D(K)$ be a connected, alternating diagram.

\begin{enumerate}

\item Find the smallest $m$ for which the projection $D(K)$ contains an $m$-gon. 
\item \begin{enumerate}
\item If $m=1$, then we resolve the corresponding crossing so that the $1$-gon becomes a state circle.

Suppose that $m=2$. Then  $D(K)$ contains twist regions with more than one crossings.
Pick $R$ to be such a twist region with $c_R>1$ crossings and $c_R-1$ bigons. Resolve all the crossings of $R$
in such a way so that all these bigons become state circles.
Create one branch of the algorithm for each bigon on $D(K)$. \label{step2a}

\vskip 0.02in

\item Suppose $m>2$. Then  by Lemma  \ref{triangle}, we have $m = 3$. Pick a triangle region on $D(K)$.
Now the process has two branches: For one branch 
we resolve each crossing on this triangle's boundary so that the triangle becomes a state  circle. For the other branch, we resolve each of the crossings the opposite way.  See Figure \ref{3-gon}.
\label{step2b}

\begin{figure}[ht]
\centering
\includegraphics[scale=.15]{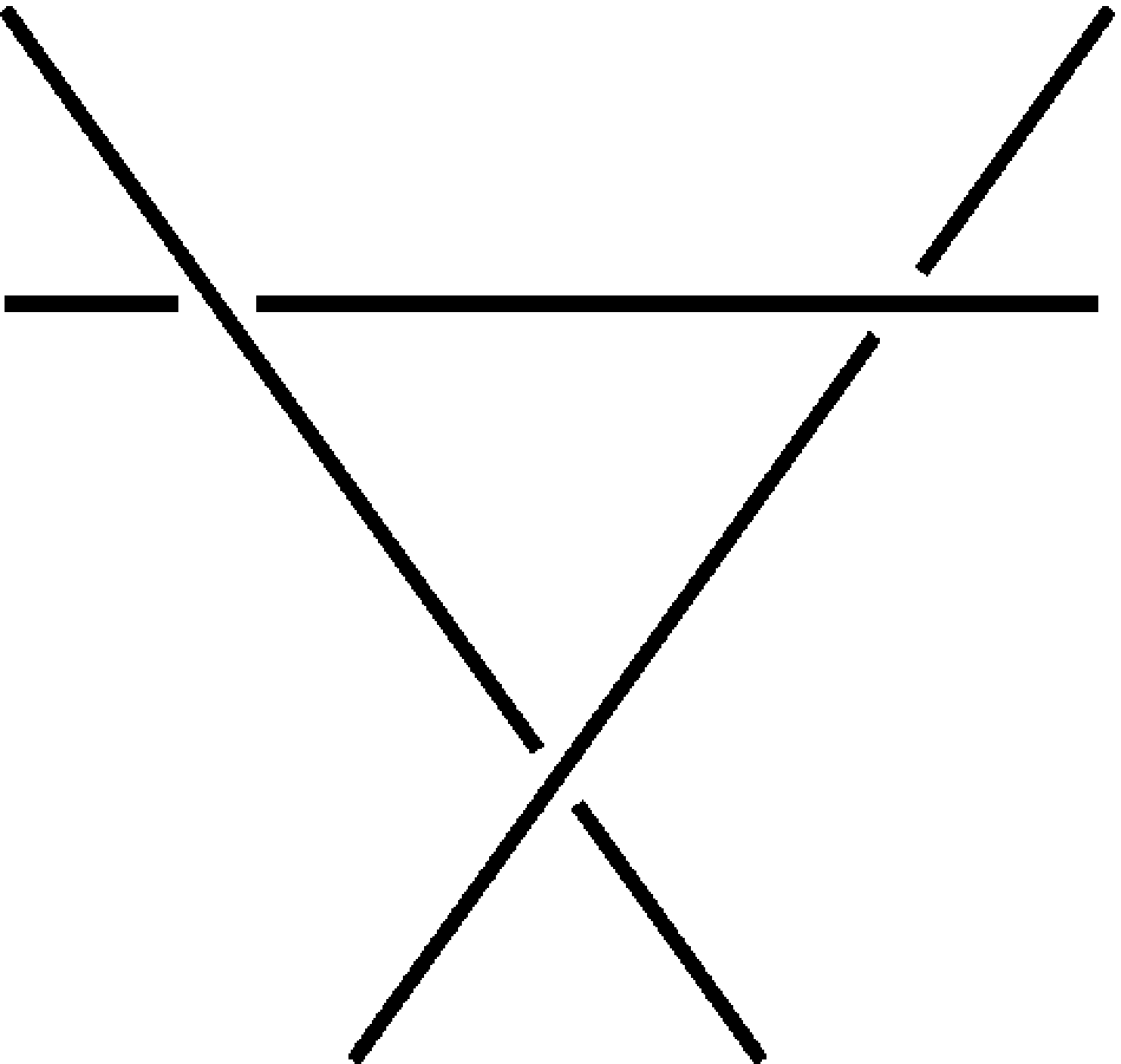}
\hspace{1cm}
\includegraphics[scale=.15]{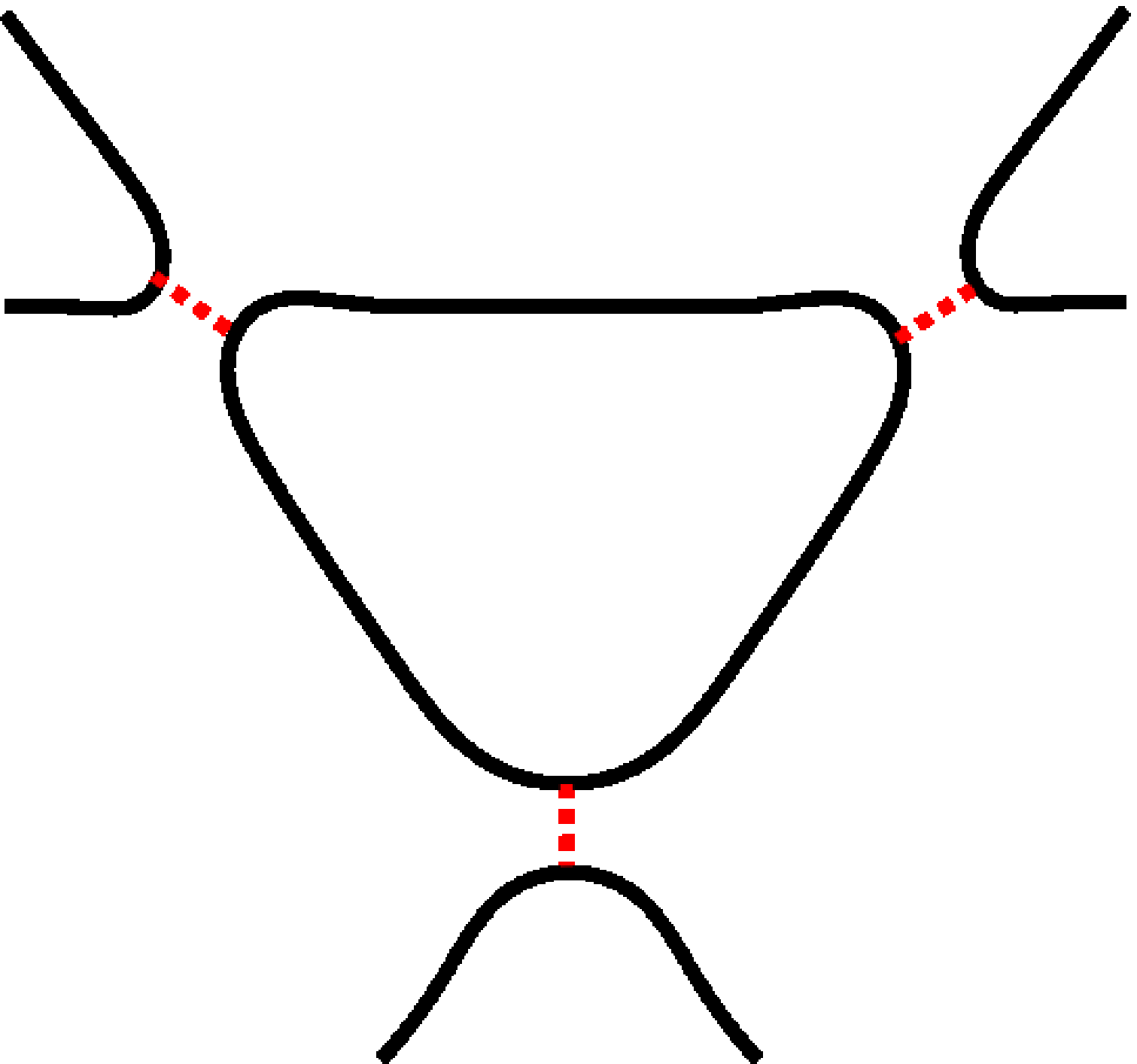}
\hspace{1cm}
\includegraphics[scale=.15]{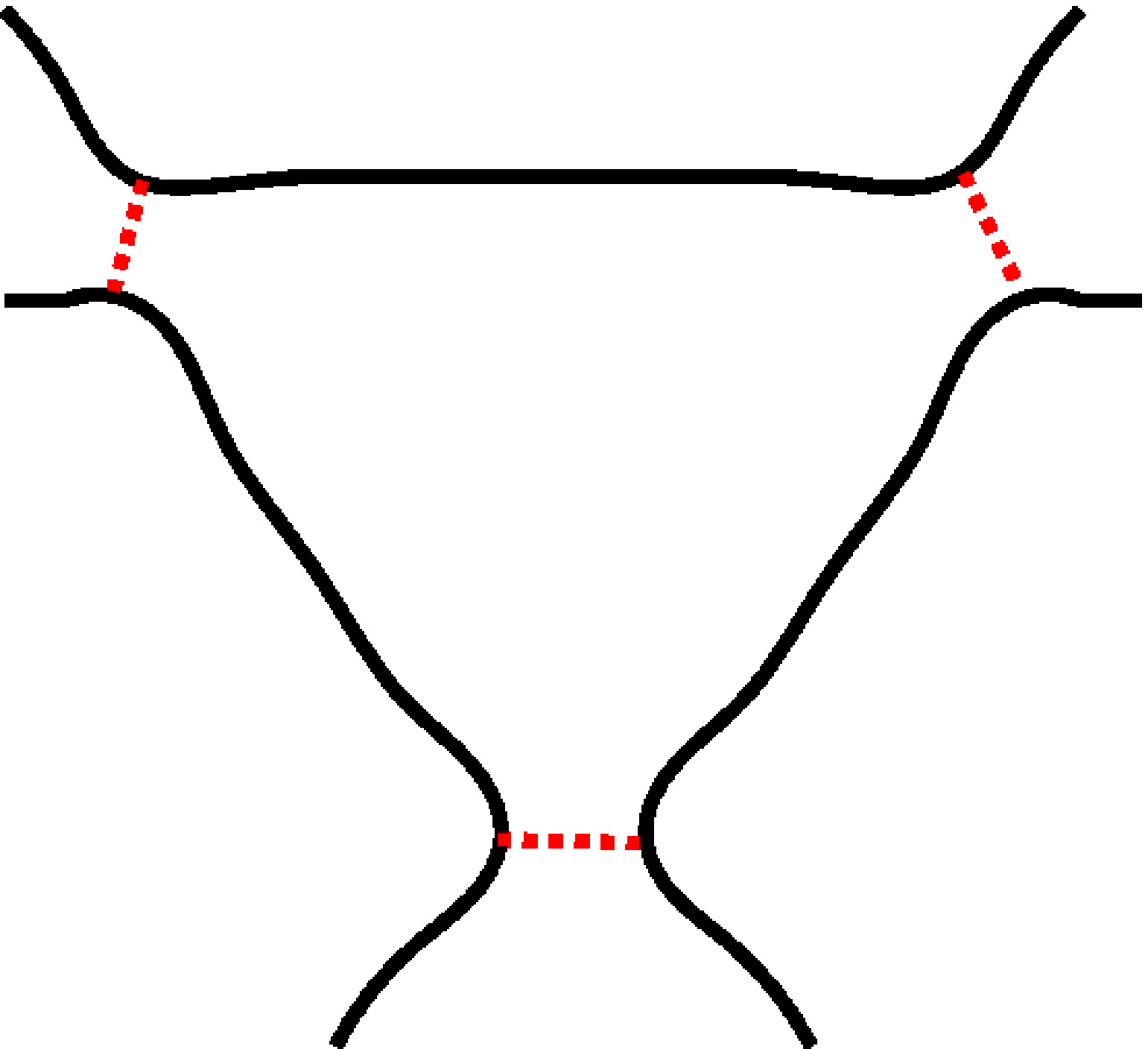}
\caption{One branch of the algorithm resolves the crossings so that the  triangle becomes a state circle. The other resolves them the opposite way.}
\label{3-gon}
\end{figure} 

\end{enumerate} 
\item Repeat  Steps 1  and 2 until each branch reaches a projection without crossings. Each branch corresponds 
to a Kauffman state of $D(K)$  for which there is a corresponding state surface.
Of all the branches 
involved in the process choose  one that has the largest number of state circles. The surface corresponding to this 
state has maximal Euler characteristic over all the states corresponding to $D(K)$.
Note that, {\emph{a priori}}, more than one branches of the algorithm may lead to surfaces of maximal Euler characteristic. 
That is, there might be several state surfaces that have the same (maximal) Euler characteristic.
\end{enumerate} }
\end{algorithm}

The following result allows to calculate  the crosscap number of an alternating link from the surface obtained by
applying
Algorithm \ref{AKalgor}
to any  alternating projection of the link.

\begin{theorem}{\rm{ \cite[Corollary 6.1]{Adamsstate}} }\label{AK}  
Let $S$ be any maximal Euler characteristic  surface obtained via Algorithm \ref{AKalgor}   from an alternating diagram of $k$-component link $K$.
Let  $C(K)$  denote the crosscap number of $K$ and let $g(K)$ denote the (orientable) genus of $K$.
Then, 
\begin{enumerate}
\item If there is a surface $S$ as above that is non-orientable then $C(K)=2-\chi(S)-k$.
\vspace{0.1in}
\item  If all the surfaces $S$ as above are orientable, we have $C(K)=3-\chi(S)-k$.
Furthermore $S$ is a minimal genus Seifert surface of $K$ and  $C(K)=2g(K)+1$.
\end{enumerate}
\end{theorem}
\smallskip

\begin{example}\label{Fig8}We should  clarify that different choices of branches as well as the order in resolving bigon regions following Algorithm \ref{AKalgor} may result in different state surfaces. In particular at the end of the algorithm  we may
have  both orientable and non-orientable surfaces that share the same Euler characteristic.
We illustrate the subtlety  in the algorithm by applying it to the knot $4_1$.
\begin{figure}[ht]
\centering
\includegraphics[scale=.4]{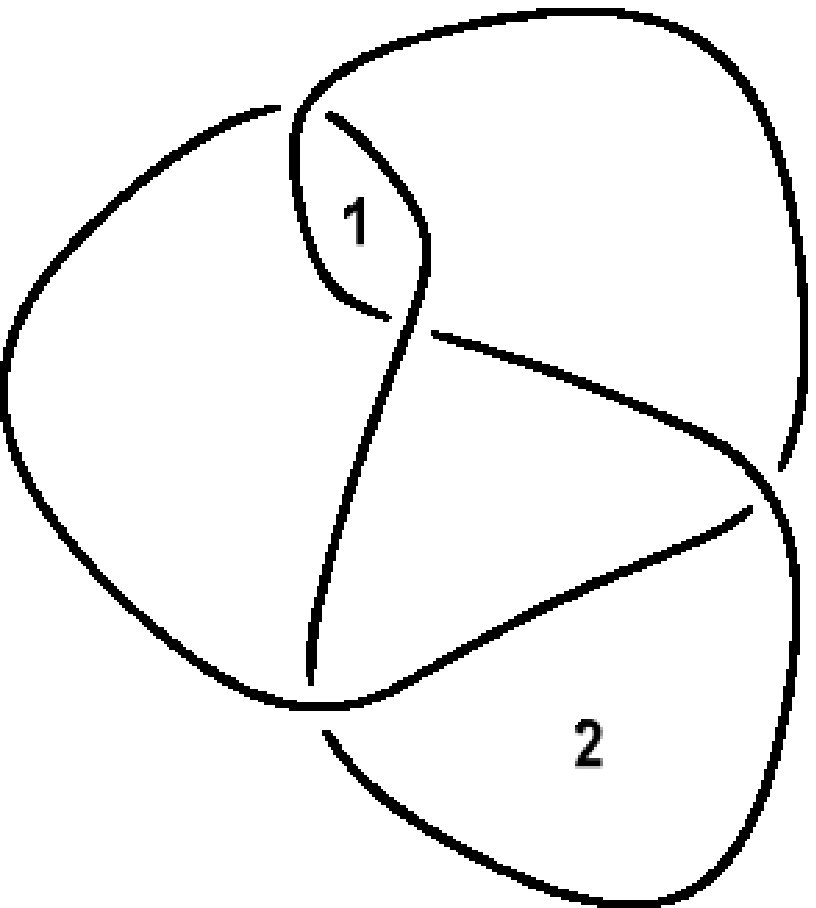} 
\hspace{2cm}
\includegraphics[scale=.4]{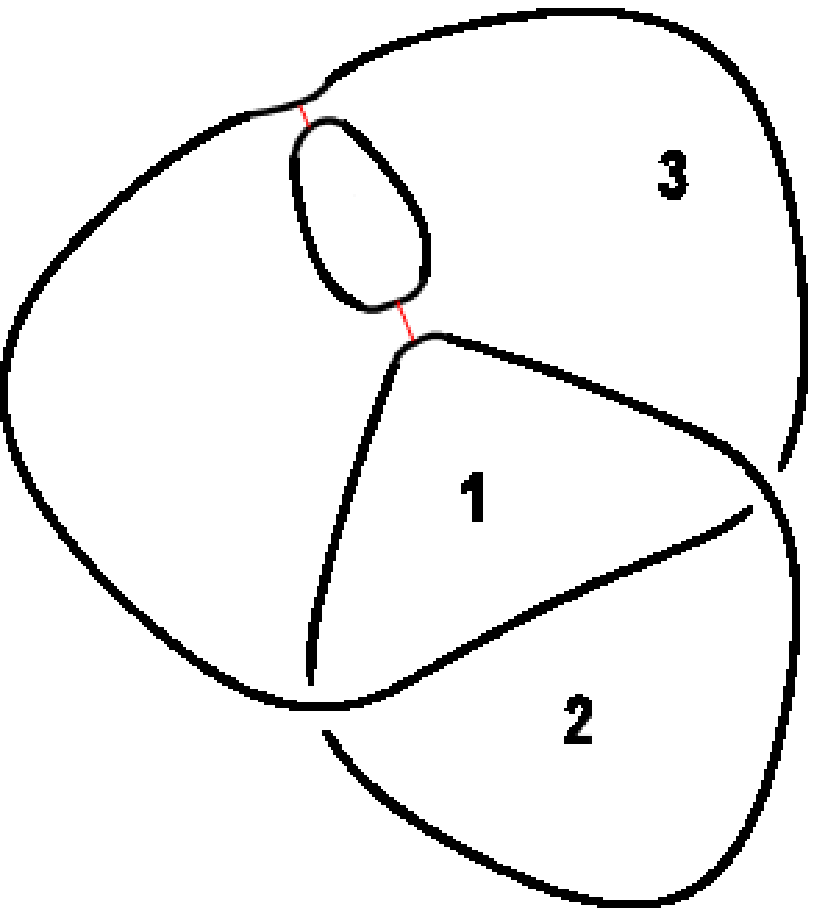}
\label{fstep1}
\caption{A diagram of the knot  $4_1$ with bigon regions labeled 1 and 2 and the diagram resulting from applying the first step of Algorithm \ref{AKalgor}.}
\end{figure}

 Suppose that we choose the bigon labeled by  1 in the left hand side picture of  Figure 7.
Then, for the next step of the algorithm, we have three choices of bigon regions to resolve, labeled by 1 and 2 and 3 in the right hand side picture of the figure.

\begin{figure}[ht]
\includegraphics[scale=.4]{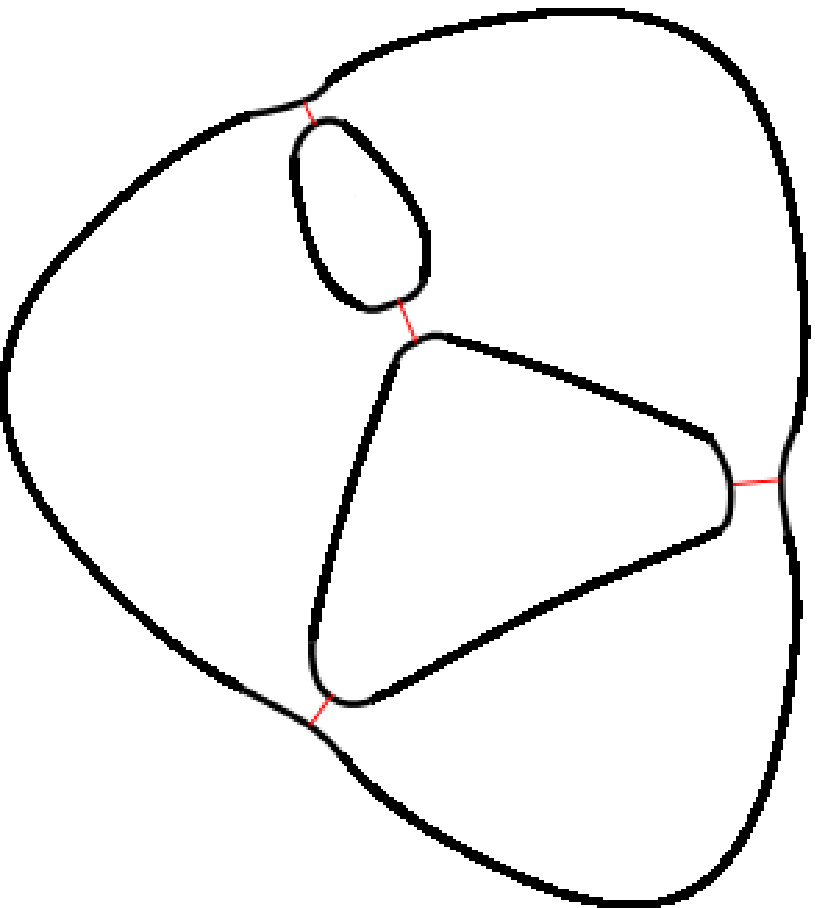} 
\hspace{2cm} 
\includegraphics[scale=.4]{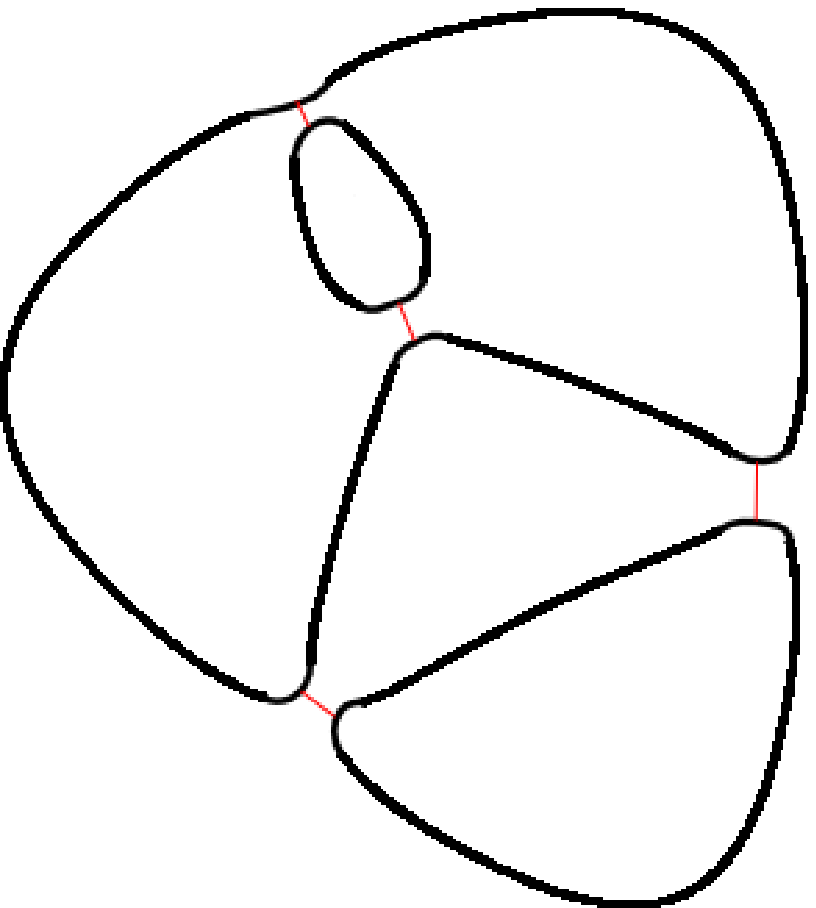}
\caption{Two algorithm branches corresponding to different bigons.}
\label{fstep2}
\end{figure}

The choice of bigon 1 leads 
to a non-orientable surface (left hand side picture of Figure \ref{fstep2}) while the choice of bigon 2 leads to an orientable surface (right hand side  picture of Figure \ref{fstep2}). 
Both of these surfaces realize the maximal Euler characteristic of -1.  The non-orientable surface realizes the crosscap number of $4_1$, which is 2.
\end{example}

\smallskip

The next lemma is important for  the results in this paper as it will allow us to apply the techniques of Section \ref{sectwo} to  obtain bounds on crosscap numbers of alternating links.

\begin{lemma} \label{incomplement} Let  $D(K)$ be a prime,
alternating, twist-reduced knot  diagram.  There is a spanning surface $S$ that is of
maximal Euler characteristic for $K$, obtained by 
applying Algorithm \ref{AKalgor} to $D(K)$, and  an augmented link $J=J_S$ such that $S$ is in the complement $E(J)$.

\end{lemma}
\begin{proof} An augmented link is obtained from $D(K)$ by adding a simple closed curve encircling each twist region. If a twist region involves more than one crossing, then there is only one way to add a crossing circle. Otherwise, there are two ways of adding a crossing circle as shown in Figure ~\ref{fig:twoways}. Let $S$ be a state surface with maximal Euler characteristic obtained by applying Algorithm \ref{AKalgor} to $D(K)$. We will show that there is such an $S$ so that  we  can augment $D(K)$ by making a choice of a crossing circle for each twist region involving a single crossing, such that $S$ lies in the complement of the augmented link. 

For a twist region $R$ involving more than one crossing, and thus consists of a  string of complementary bigon regions arranged end to end, we augment by adding a crossing circle $C_R$ encircling the twist region. Since $D(K)$ is prime, none of its complementary regions can be an $1$-gon. In this case, the algorithm picks the resolution of the crossings of $R$ so that
each bigon  becomes a  state circle following Step ~\ref{step2a}. In any state surface which have this resolution at the crossings of $R$, these bigon disks are joined with twisted strips, and the crossing circle $C_R$ encloses the twisted strips. We may arrange
so that each twisted strip intersects the crossing disk corresponding to $C_R$ only in its
interior.  Hence, the portion of any state surface obtained by the algorithm involving a twist region containing at least one bigon, will lie in the complement of  $C_R$. See Figure ~\ref{bigonstrips}.  
\begin{figure}[ht]
\includegraphics[scale=.1]{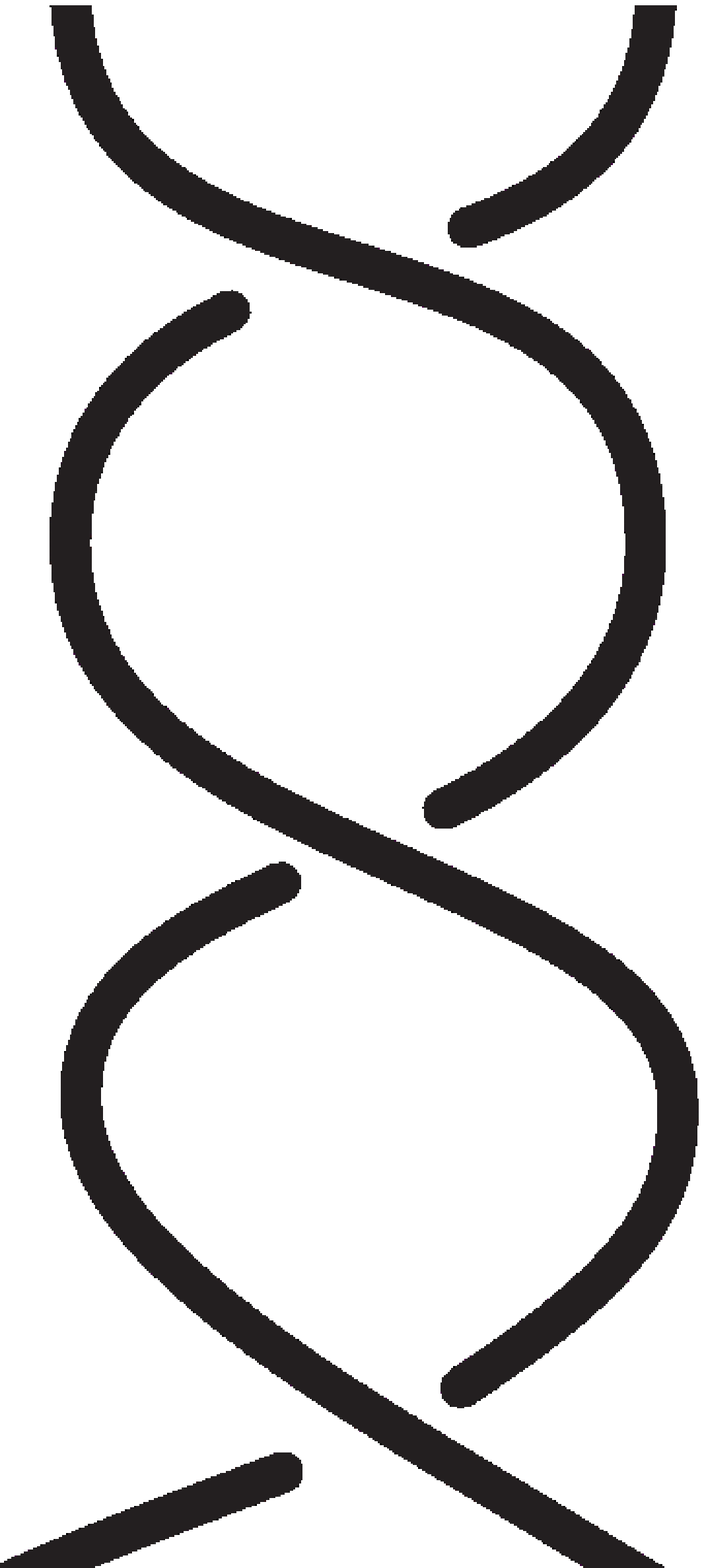}
\hspace{2cm} 
\includegraphics[scale=.1]{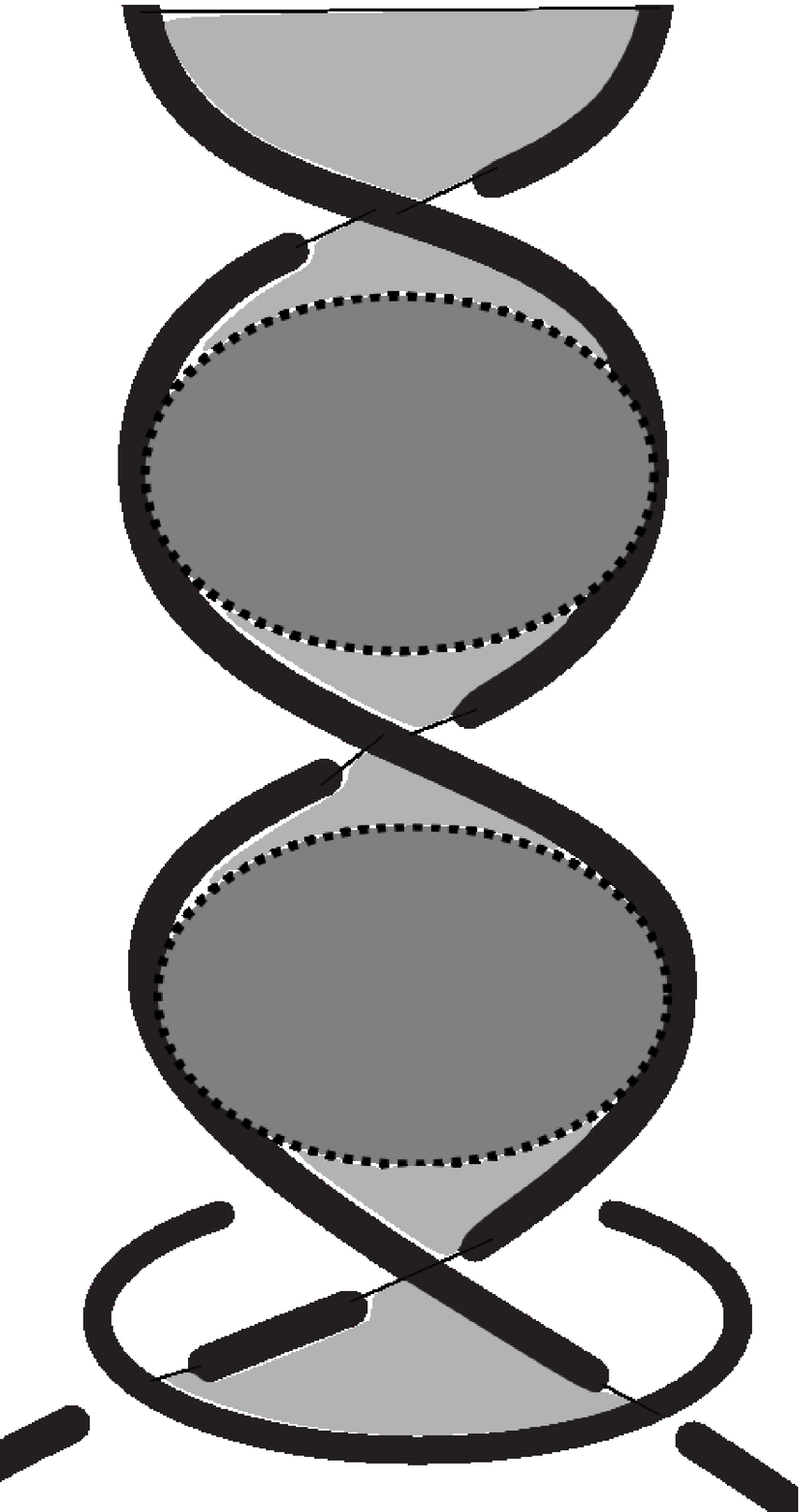}
\caption{ The portion of $S$ through a twist region with more than one crossing and the crossing circle for the twist region.}
\label{bigonstrips}
\end{figure} 
 
Run Algorithm \ref{AKalgor} till all the twist regions involving more than one bigon have been resolved and augmented as above.
Let 
$D'$ be one of the resulting link diagrams at this stage of the algorithm. If there is a bigon on $D'$ that comes from a bigon in the original diagram $D(K)$ then we apply step 1
of Algorithm \ref{AKalgor} to this bigon and we add an augmentation component the same way as in Figure \ref{bigonstrips}.

Suppose that $D'$ contains a bigon that was not a bigon on $D(K)$. Since $D(K)$ is twist reduced, such bigons in $D'$ can only come
from triangles in $D(K)$ that had a twist region with more than one crossings attached to them.
If there is such a bigon in $D'$,  Algorithm \ref{AKalgor} will apply Step ~\ref{step2a} to these crossings such that the bigon becomes a state circle. 
Then we choose an augmentation for each of the two crossings by putting two crossing circles, one for each of the two crossings resolved, so that each crossing circle encloses a twisting strip from the resolution at one crossing. The portion of the state surface coming from resolving these two crossings will then also be disjoint from the crossing circles. Repeat this procedure following the  algorithm until there are no more bigons on the resulting diagrams $D'$.

 Now each of the crossings  on $D'$ corresponds to a twist region of $D(K)$ containing a single crossing. We decide on which way to add a crossing circle to each of these twist regions. 
If there are no more crossings left in the projection, then we are done. Otherwise, we have a projection for whom a minimal $n$-gon is a triangle by Lemma ~\ref{triangle}, and Step ~\ref{step2b} of the  algorithm is applied to resolve each of the three crossings of the triangle, see Figure ~\ref{3-gon}. In addition, each of the remaining crossings is a twist region in the original projection $D(K)$. Let $S$ be a maximal Euler characteristic surface obtained from the algorithm. Each time that Step ~\ref{step2b} is applied, a branch of the algorithm is chosen to generate $S$. We augment each of the three crossings of the triangle based on which branch is chosen. In either branch, we add a crossing circle for each of the three crossings, such that it encircles the crossing strip from the resolution of the chosen branch. See Figure ~\ref{branch}. 

\begin{figure}[ht]
\includegraphics[scale=.22]{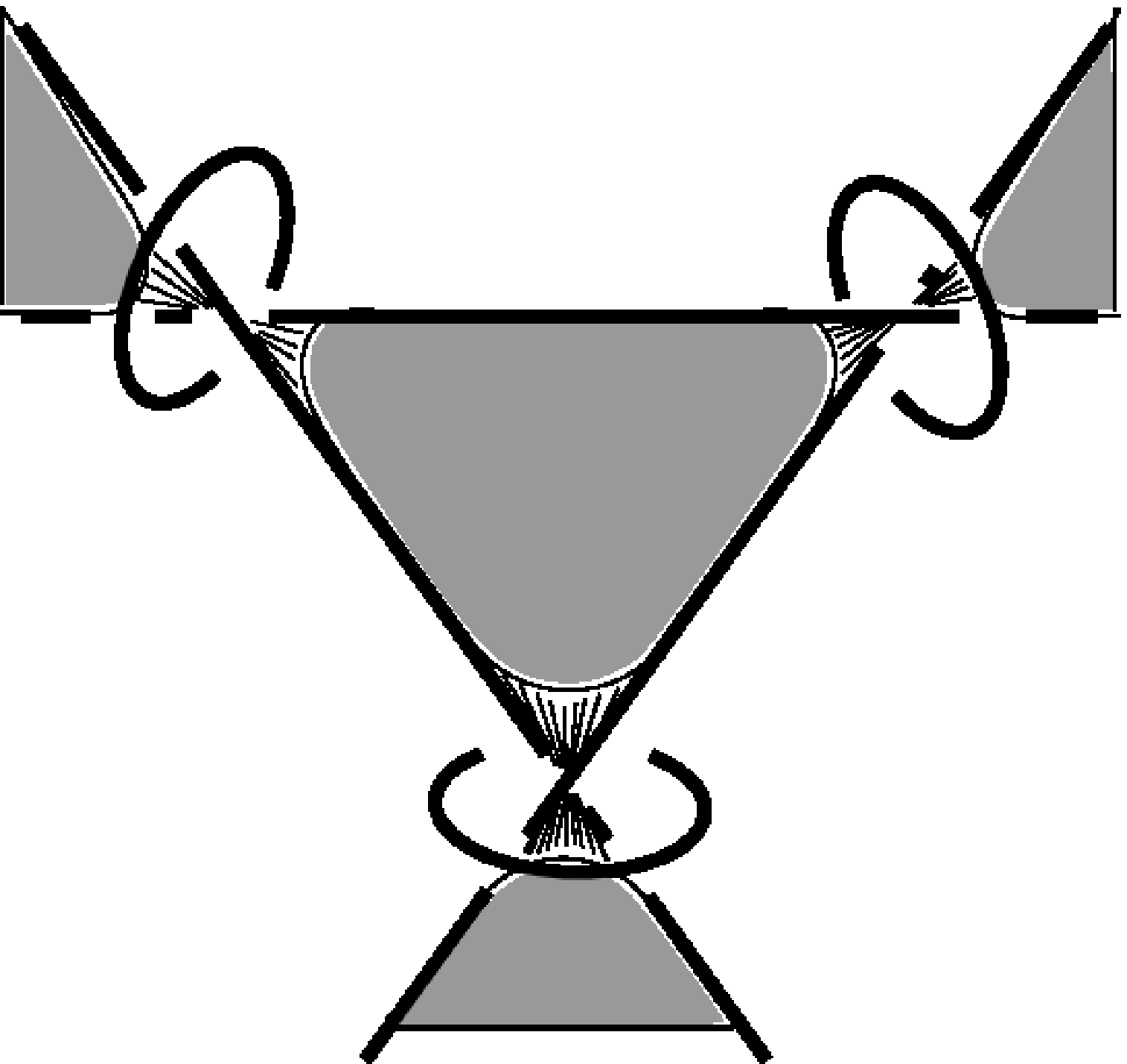}
\hspace{3cm}
\includegraphics[scale=.22]{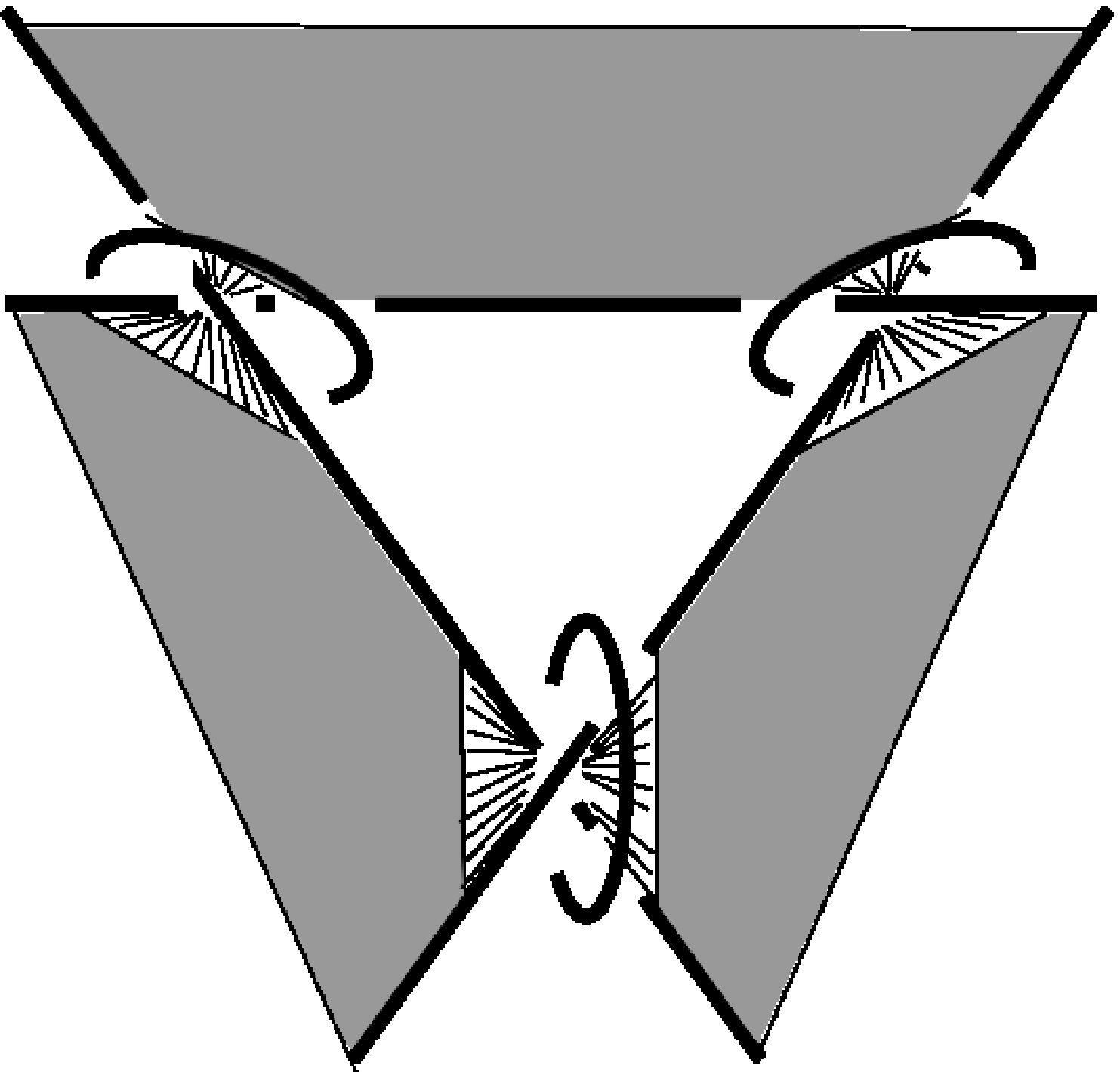}
\caption{Surfaces from two branches of the splitting in Step 2b in the Adam-Kindred algorithm and the corresponding choice of augmentation.
}\label{branch} 
\end{figure} 

We repeat as in Step 3 of  Algorithm \ref{AKalgor} which will make a decision for splitting at the rest of the crossings. If a bigon is encountered again as a minimal $n$-gon we repeat the procedure for a bigon  region where each of the crossings belongs to a twist region in $D(K)$. Otherwise we apply the procedure for when a minimal $n$-gon is a triangle. We stop when there are no more twist regions in $D(K)$ to augment.
\end{proof}

\subsection{Normalization and crosscap number estimates} 
Here we use the results of Section \ref{sectwo} to obtain two-sided bounds of
the crosscap number of an alternating link in terms of the twist number of any prime, twist-reduced alternating link diagram. 

We begin with the following lemma that shows that the crosscap number of an alternating link is always bounded by the twist number of any alternating projection  from above; that is, primeness and twist-reducibility is not needed for this part.

\begin{lemma} \label{uppert}
  Let $K \subset S^3$ be a link of $k$ components with an alternating diagram  $D(K)$ that has $t \geq 2$ twist
  regions. Let $C(K)$ denote the crosscap number of $K$. We have

$$  C(K) \, \leq \, t+2-k.
  $$
\end{lemma}
\begin{proof} Let $S$ be a surface obtained by applying Algorithm \ref{AKalgor}  to $D(K)$ and
let $\sigma$ denote the   Kaufmann state  of $D(K)$ to which 
$S$ corresponds.
Let $v_{b}$ denote the number of state circles that are bigons and let $v_{nb}$ denote the non-bigon
state circles. Also let $c_1$ denote the number of crossings in $D(K)$ each of which forms its own 
twist region and let $c_2$ denote the remaining crossings. Since we assume that there are at least two twist regions we have $v_{nb}\geq 1$. Thus
we have
  {\setlength\arraycolsep{3pt}
\begin{eqnarray*}
- \chi(S)=
&=& c_2-v_b+c_1-v_{nb} \\
&\leq& t-1.\\
\end{eqnarray*}}
Now the upper bound follows at once from Theorem \ref{AK}.
\end{proof}

 Before we are able to estimate $C(K)$ from below we need some preparation.
Let  $D(K)$ be  a link diagram and let $J$ be an augmented link obtained from $D(K)$
with $L$ the corresponding fully augmented link. Suppose that the exterior $E(J)\cong E(L)$ is hyperbolic with an angled polyhedral decomposition $\{P_1, P_2\}$ as described in Section
\ref{sectwo}. Let $S$ be a spanning surface of $K$ that realizes $C(K)$.
Having Lemma \ref{incomplement} in mind, we will also assume that 
$S$ is disjoint from the crossing circles of $J$; hence we can view it as a surface in $E(J)\cong E(L)$.

We wish to apply Theorem \ref{chi-etimate} to estimate $\chi(S)$.  
In order to do so
we need to have $S$ in normal position with respect to the polyhedral decomposition $\{P_1, P_2\}$. 
The standard argument of making an surface that is incompressible and $\partial$-incompressible 
(that is, \emph{essential}) normal  in a triangulated 3-manifold  \cite{haken},
can be adjusted to work  in the setting of  more general polyhedral decompositions. The argument is written down by Futer and Guéritaud
\cite[Theorem 2.8]{fg:arborescent}.   Note however  that surfaces that realize $C(K)$ need  not be essential in $E(K)$. For instance,  if all the state surfaces obtained from Algorithm \ref{AKalgor} applied to  an alternating projection $D(K)$
are orientable, then a spanning surface that realizes $C(K)$ is obtained from a minimal genus Seifert surface by adding a half-twisted
band. Such a surface is $\partial$-compressible.
This, for example, happens for the knot  $7_4$  \cite{Adamsstate}.

For our purposes we are only interested in the question 
of whether $S$  can be converted to a spanning surface  that is normal with
respect to the polyhedral decomposition of $E(L)$, without changing $\chi(S)$ and the surface orientability.
For this we examine how a spanning surface  $S$ that realizes  $C(K)$ behaves under the general  process that converts any properly embedded
surface in $E(J)$ into a normal one with possibly different topology \cite{Jaco-Rub}.
We have the following lemma that applies beyond the class of alternating links and might be of independent
interest.

\begin{lemma}\label{normal} Let $K \subset S^3$ be a   link with a  prime, 
  twist-reduced diagram  $D(K)$. Suppose that $D(K)$ has $t \geq 2$ twist
regions.  Suppose that there is a spanning surface $S$  in the exterior $E(J)\cong E(L)$ of an augmented link of $D(K)$ and  such that $C(S)=C(K)$.
Then exactly one of the following is true:
\begin{enumerate}
\item There  is a non-orientable,  spanning surface $S' \subset E(L)$ for  $K$, that  is in normal form and
such that $\chi(S)=\chi(S')$.
\item We have $C(K)=2g(K)+1$.
Furthermore, there is  a Seifert surface of $K$ that lies in $E(L)$ such that it realizes $g(K)$ and it is in normal form.
\end{enumerate}
\end{lemma}  
\begin{proof} 

By assumption $S$ realizes $C(K)$.  Thus $S$ has maximal Euler characteristic among all non-orientable spanning surfaces
of $K$.
As discussed above, we will assume that $S$ is not  necessarily essential and examine how compressions
and $\partial$-compressions may interfere with a process of converting $S$  to a normal surface.
 Examining the moves required during this process \cite[Theorem 2.8]{fg:arborescent}, and since $S$ contains no closed components, we see that there are two situations to consider:
\begin{enumerate} 
\item $S$   admits a compression disk       
$D$, that lies in the interior of a single face
of a polyhedron $P\in \{P_1, P_2\}$. 

\item The intersection of  $S$ with  a face of a polyhedron $f\subset \partial P$ is an arc $\gamma$ that runs from an edge of $e\subset \Gamma \subset \partial P$  to itself, or from an interior edge to an adjacent 
boundary edge.
\end{enumerate}

In (1) there are two cases to consider according to whether   $\partial D$ separates $S$ or not.
First suppose that $\partial D$  is non-separating on $S$.
Compressing along $D$ we obtain  a spanning surface $S'$ for $K$ with $\chi(S')=\chi(S)+2$.
Since $S$  realizes $C(K)$,  $S'$ cannot be non-orientable. Thus we have an orientable
spanning surface of $K$; that is a Seifert surface.  Adding a half-twisted band to $S'$ (i.e. adding a crosscap)
produces a non-orientable spanning  surface $S_1$ of $K$ with $\chi(S_1)=\chi(S')-1$. Thus,
$S_1$ is a non-orientable spanning  surface with $\chi(S_1)=\chi(S)+1> \chi(S)$ which contradicts the fact that
$S$ realizes $C(K)$. Thus this case will not happen.

Suppose now that $\partial D$ separates $S$.
We will look at  such a disk so that $\partial D$ is innermost in the sense
 that one of the
components of $S\setminus \partial D$ lies entirely in a single polyhedron $P$.
Compressing along $D$ gives two surfaces $S_1$ and $S_2$ with
$\chi(S)=\chi(S_1)+\chi(S_2)-2$. 

Suppose that both surfaces have non-empty boundary.
Then the disjoint union of $S_1, S_2$ is a non-orientable  spanning surface of $K$ with Euler characteristic
$\chi(S_1)+\chi(S_2)>\chi(S)$, contradicting the assumption that $S$ realizes $C(K)$.
Thus, one of  $S_1, S_2$, say  $S_1$, must be a closed surface
and all of $\partial S$ is left on $S_2$.
Since $S_1$ is a closed surface embedded in $S^3$, it is orientable. Hence $S_2$ is a non-orientable.
 Since $S$ realizes $C(K)$ and $\chi(S_1)\leq 2$, it follows that
 $\chi(S)=\chi(S_2)$.
Thus we may ignore $S_2$, replace
$S$ with $S_2$ and continue with the normalization process.

Next we treat case (2): The arc $\gamma$ cuts off a disk $D\subset f$, with $ \partial D$ consisting
of $\gamma$  and an arc that lies on the boundary of $f$. By an innermost argument, we may assume
that the interior of $D$ is disjoint from $S$.   Again, following the argument  in the proof of
\cite[Theorem 2.8]{fg:arborescent}, if
$\gamma$ runs from an interior edge to an adjacent boundary edge, 
the disk $D$ will guide an isotopy of $S$ that can be used to eliminate the arc $\gamma$ \cite[Figure 2.2]{fg:arborescent}.
Similarly, if $e$ is an interior edge or $e$ is a boundary edge and $f$ is a boundary face, we can use $D$ to obtain an isotopy that
will eliminate $\gamma$ and decrease 
the number of intersections of $S$ with $\Gamma$ (compare, left panel of \cite[Figure 2.1]{fg:arborescent}). It follows, that the only case left to examine 
is
when  $e$ is a boundary edge and $f$ is an interior face of $S$.
In this case $D$ is a $\partial$-compression disk of $S$.  Consider the arc
$\delta:= \partial D\setminus \gamma \subset e$ and  let $T$ denote the boundary component of $\partial E(K)$ containing it.
There are two cases to consider: (i) $\delta$ cuts off a disk in the annulus $T\setminus \partial S$; and (ii)
$\delta$ runs from one component of  $T\setminus \partial S$.
We will perform surgery  ($\partial$-compression)  along  $D$. This may cut $S$ into more components or not according to
whether we are in case (i) or (ii).
Surgery along a $\partial D$ doesn't change the total geometric intersection number of $\partial S$ with the meridians of $\partial E(J)$. Thus, by Lemma \ref{spanning}, it will produce a 
spanning surface of $K$ and possibly some (redundant) components.

First suppose  that surgery along $D$ splits $S$ into two surfaces $S_1$ and $S_2$. 
Suppose that both surfaces have non-empty boundary.
Then the disjoint union of $S_1, S_2$ is a non-orientable  spanning surface of $K$ with Euler characteristic
$\chi(S_1)+\chi(S_2)>\chi(S)$, contradicting the assumption that $S$ realizes $C(K)$.
Thus all
 the components of $\partial S$ disjoint from $D$ must  remain on one of $S_1, S_2$, say on $S_2$ and the intersections of $\partial S$ with the meridians of $\partial E(K)$, also  remain on $\partial S_2$. 

Thus $S_1$ has one boundary component that
is either homotopically  trivial on $\partial E(K)$ or isotopic to a meridian of $\partial E(K)$. In either case $\partial S_1$ bounds a disk  in $S^3$. We may cap $\partial S_1$ 
with this disk to
produce a closed surface embedded in $S^3$, which must be orientable. Hence $S_2$ is a non-orientable spanning surface
for $K$. 
 But since $S$ realizes $C(K)$ and $\chi(S_1)\leq 1$
we must have $\chi(S)=\chi(S_2)$. 
 Thus we may ignore $S_1$, replace
$S$ with $S_2$ and continue with the normalization process.

Next suppose   that surgery along $D$ doesn't disconnect $S$. Then we get  a spanning surface  $S'$, with $\chi(S')=\chi(S)+1$.
Since $S$ was assumed  to realize $C(K)$,  $S'$ cannot be non-orientable. Thus we have an orientable
spanning surface of $K$.  We claim that $S'$ must be a minimal genus Seifert surface, that is $g(S')=g(K)$.
For,  suppose that $K$ has  a Seifert surface $S_1$ with $\chi(S_1)>\chi(S')$. Then adding a half-twisted band to $S_1$
would give a non-orientable spanning surface $S''$ with $\chi(S'')=\chi(S_1)-1>\chi(S')-1=\chi(S)$, contradicting the fact that
$S$ realizes $C(K)$. Thus, $S'$ is a minimal genus Seifert surface of $K$  that lies in $E(L)$; that is we have $g(S')=g(K)$.
Now it is clear that the surface obtained by a half-twisted band to $S'$ is a non-orientable spanning surface $S''$ of $K$ that has maximal Euler characteristic
among all such surfaces. Thus we have $C(K)=C(S'')=2g(S')+1=2g(K)+1$. Since $S'$ is minimal genus and orientable,
it is incompressible and $\partial$-incompressible in $E(K)$ and thus in $E(L)$. Hence we may isotope $S'$ to be normal with respect to the polyhedral decomposition  \cite[Theorem 2.8]{fg:arborescent}.
\end{proof}

 Now we are ready to prove the following.

\begin{theorem} \label{chi-etimate1}
  Let $K \subset S^3$ be a link of $k$ components with a prime,
  twist-reduced alternating diagram  $D(K)$. Suppose that $D(K)$ has $t \geq 2$ twist
  regions. Let $C(K)$ denote the crosscap number of $K$. We have

$$ \left\lceil \frac{t}{3} \right\rceil\, + \, 2-k \, \leq C(K) \, \leq \, t+2-k,
  $$
where $\lceil \cdot \rceil$ is the ceiling
  function that rounds up to the nearest larger integer. Furthermore, both  bounds are sharp.

\end{theorem}
\begin{proof} 
The upper bound comes from Lemma \ref{uppert}. To derive the lower bound,
let $S$ be a state surface obtained from Algorithm \ref{AKalgor}  to $D(K)$
and let $J$ be the augmented link of Lemma \ref{incomplement}, with $L$ the corresponding fully augmented link. By Theorem
\ref{AK}, and the proof of Lemma \ref{incomplement} one of the following is true:
\begin{enumerate}

\item $S$ is non-orientable  and realizes $C(K)$; that is we have  $C(K)=C(S)$.
\item $S$ is orientable and we have $C(K)=2g(S)+1$.
\item $S$ is orientable and there is a non-orientable spanning surface of $K$ with the same Euler characteristic. 
\end{enumerate}
Suppose we are in case (1).  Then by Lemma \ref{normal} we may replace $S$ by a spanning surface that also realizes $C(K)$ and is normal
with respect to the polyhedral decomposition of $E(L)$. Theorem \ref{chi-etimate}
gives $$C(K)=C(S)\, =\,  2\, -\, \chi(S)\, -\, k\, \geq \,  \left\lceil \frac{t}{3}\right\rceil\, +2\, -k,$$
and the lower bound follows.

On the other, hand if we are in (2) then $S$ is a minimal genus Seifert surface of $K$ and thus
 it is incompressible ands $\partial$-incompressible. Thus we may again replace $S$ with one  into normal form.
Since $S$ is disjoint from the crossing circle components of $J$,
Theorem \ref{chi-etimate}
gives $$C(K)=2g(S)+1=2-\chi(S)-k+1\geq  \left\lceil \frac{t}{3} \right\rceil \,+\,3\,  -\, k.$$ Hence the lower bound follows again.

Finally in Case (3) $S$ is an essential surface and, again, there is an orientable  normal surface of the same Euler characteristic.
Now $C(K)=C(S)\, =\,  2\, -\, \chi(S)\, -\, k$ and the lower bound is obtained exactly as in the case of (1).

\begin{figure}[ht=1in, width=1in]
\includegraphics[scale=.20]{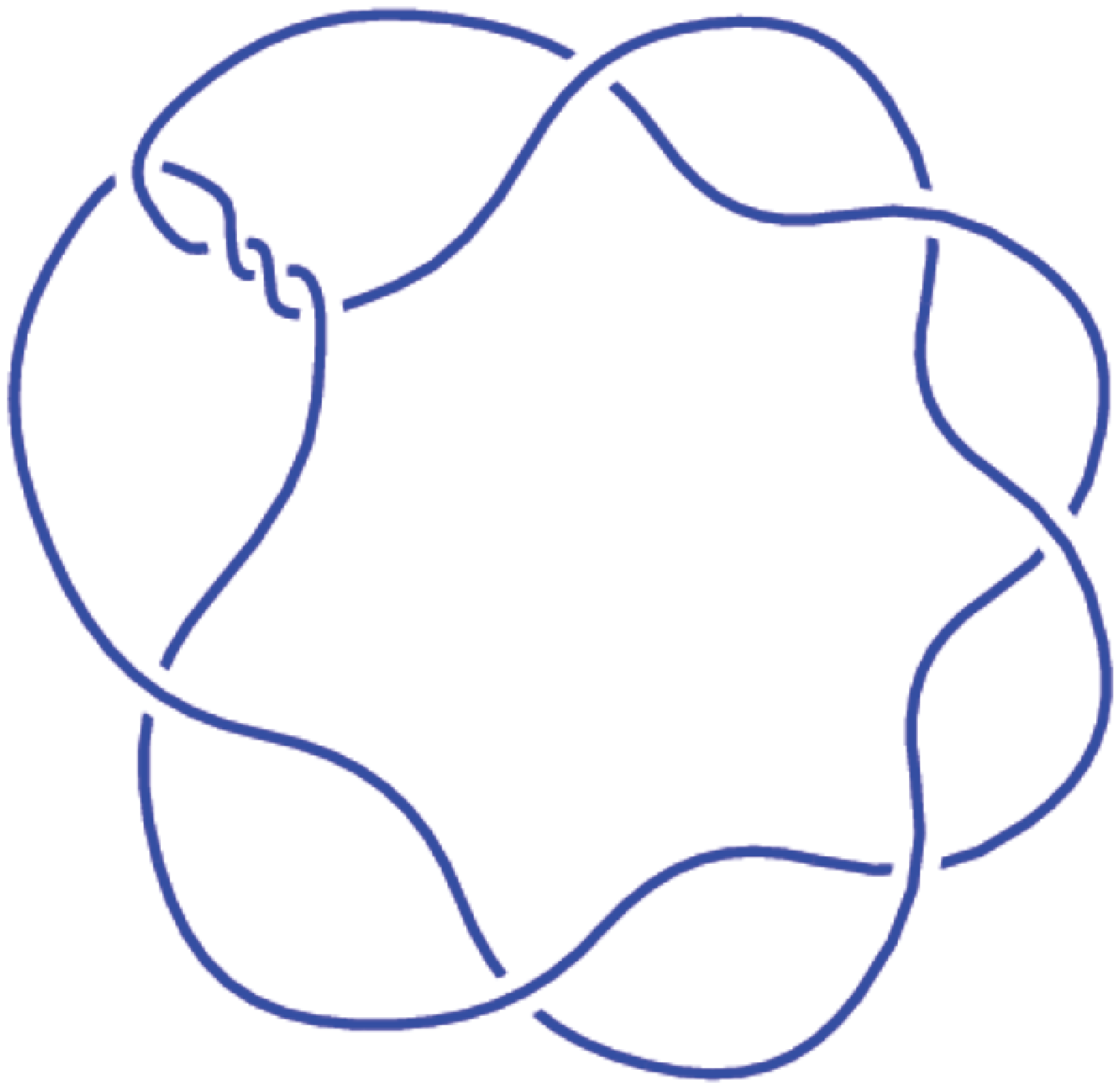}
\hspace{1.0cm} 
\includegraphics[scale=.20]{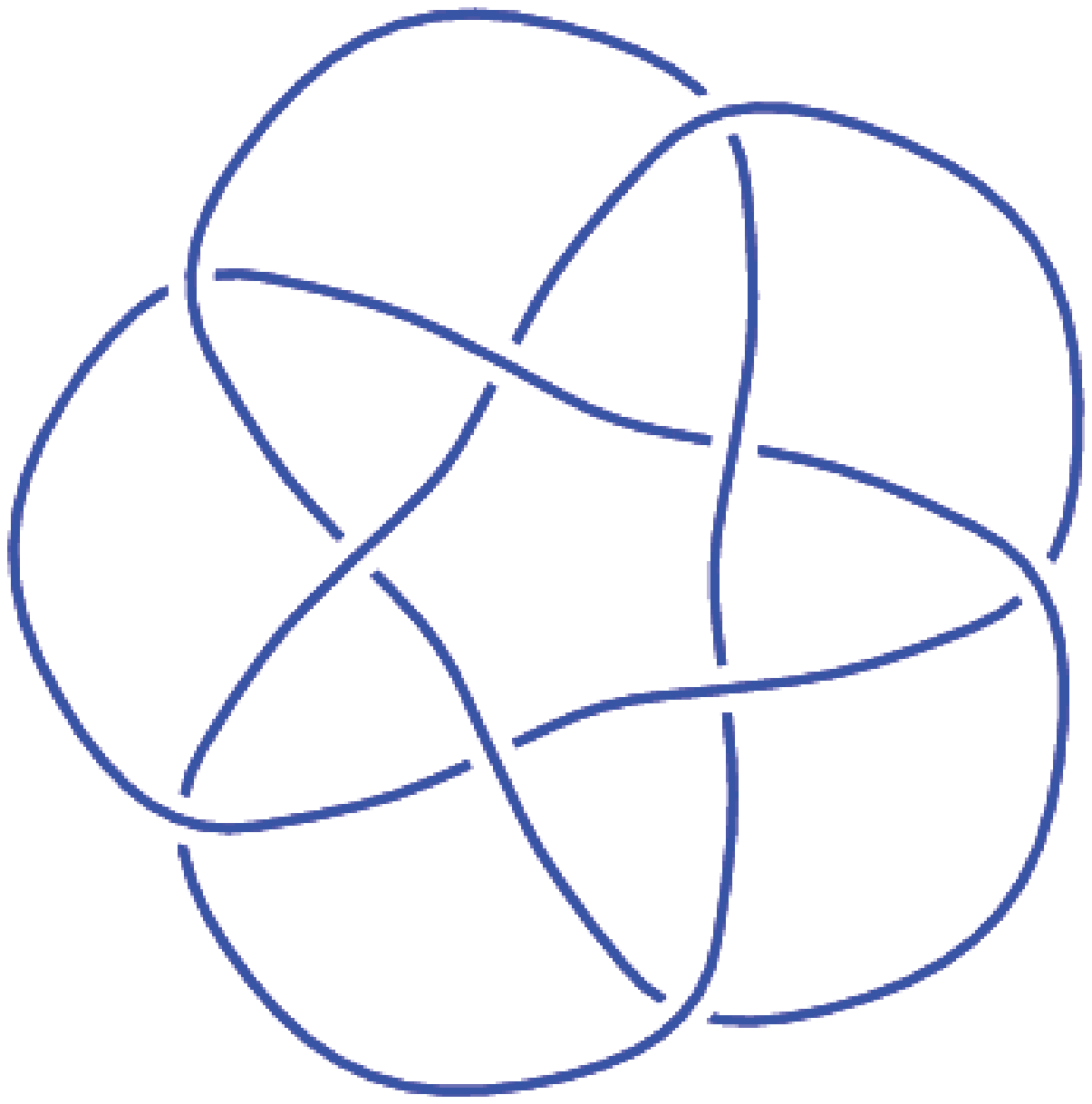}

\caption{ The knots $10_3$ (left) and  $10_{123}$ (right).}
\label{examplec=t}
\end{figure}

It remains to prove that both bounds are sharp:
Consider the alternating knot $10_3$ of Figure \ref{examplec=t}.
The twist number of the diagram shown there is $t=2$ and each twist region consists of more than one crossing.
All the  branches of Algorithm \ref{AKalgor} are seen to give an orientable surface of genus 1.
Thus, by Theorem \ref{AK}, $C(K)=2g(K)+1=3=t+1$. The same argument applies to the knots obtained by adding any even number of crossings in each twist region of
the knot $10_3$. Hence we have an infinite family of alternating knots with $C(K)=2g(K)+1=3=t+1$.

To discuss some examples where our lower bound is sharp, note that if  a  knot $K$   processes an alternating diagram with $c=t$,
then we have

$$1+ \left\lceil \frac{ c}{3}\right\rceil \; \leq \; C(K) \; \leq \;  
\left\lfloor{ \frac{c}{2}}\,\right\rfloor.
$$

Now observe that, for instance, if  $c=t=13$, then $C(K)=6$. Similarly,
 if $c=t=10$, then $C(K)=5$.  A concrete example, is  the knot $10_{123}$ shown in Figure \ref{examplec=t}.
More examples where the lower bound is sharp are discussed in Section \ref{calculations}.
 \end{proof}

Now we explain how Theorem \ref{thm:chi-etimate2}, stated in the Introduction, follows from Theorem  \ref{chi-etimate1}:
For $k=1$, the inequality of Theorem  \ref{chi-etimate1}
becomes
$$ \left\lceil \frac{t}{3}\right\rceil\, + \, 1  \;\leq \; C(K) \; \leq \; t+1.
  $$

Thus the lower bound follows.
Murakami and Yasuhara \cite{upper}
 showed that for a knot $K$ with a connected, prime diagram of $c$ crossings, we have
$$C(K)\leq \left\lfloor{ \frac{c}{2} }\right\rfloor.
$$
Thus the upper bound follows from these two inequalities.

\subsection{Jones polynomial bounds}Now we discuss how the results stated in the introduction follow from the above results. 
Recall that for a knot $K$, we have defined $T_K:= \abs{\beta_K} + \abs{\beta'_K},$
where $\beta_K$ and $\beta'_K$ denote the second and the penultimate coefficients
of the Jones polynomial of $K$, respectively.

\begin{named}{Theorem \ref{thm:cupjones}} 
Let $K$ be a non-split, prime alternating link  with $k$-components and  with crosscap number $C(K)$.
Suppose that $K$ is not a $(2,p)$ torus link.
We have

$$ \left\lceil \frac{T_K}{3}\right\rceil \, +\, 2-k\; \leq \; C(K) \; \leq \; T_K+2-k.
  $$
Furthermore, both  bounds are sharp.
\end{named}
\begin{proof} Let  $D(K)$  be a connected, twist-reduced alternating diagram that has $t \geq 2$ twist
regions. Then, by \cite[Theorem 5.1]{dasbach-lin:volumish},  we have $t=T_K$.
A prime alternating link admits prime, twist reduced alternating diagrams; every alternating diagram can be converted to a twist-reduced one by flype moves \cite[Lemma 4]{lackenby:volume-alt}.
Thus,  the result follows  from Theorem \ref{chi-etimate1}. 
\end{proof}

Now we are ready prove 
 \ref{thm:cupjonesknots} which we restate for the convenience of the reader.

\begin{named}{Theorem \ref{thm:cupjonesknots}}
 Let $K$ be an alternating, non-torus knot  with crosscap number $C(K)$ and let $T_K$ be as above.
We have

$$1+ \left\lceil \frac{ T_K}{3}\right\rceil \; \leq \; C(K) \; \leq \;  {\rm {min}}{ \left\{ T_K+1, \ 
\left\lfloor{ \frac{s_K}{2}}\,\right\rfloor \right\}}
$$
where $s_K$  denotes the span of $J_K(t)$. 
Furthermore, both
 the upper and lower bounds are sharp.
 \end{named}

\begin{proof} Kauffman \cite{Kaufjones} showed that the degree span of the Jones polynomial of an alternating link
is equal to the crossing number of the link. Furthermore, both the degree span and the crossing number of alternating knots are known to 
be additive under the operation of connect sum \cite{lickorish:book}. Using these, the upper inequality
 follows at once from
Theorem \ref{thm:chi-etimate2}.
Furthermore,
the lower inequality follows for all prime alternating knots.

To finish the proof we need to show that the lower inequality holds for connect sums of alternating knots.
To that end let  $K\#K'$ be such a knot. By \cite{upper}, we have
$$C(K\#K^{'})\geq C(K)+C(K^{'})-1.$$
Since the
Jones polynomial is multiplicative under connect sum
we have
$T_K+T_{K^{'}}\geq T_{K\# K^{'}}$.
Hence we have
{\begin{eqnarray*}
C(K\#K^{'})
&\geq & C(K)+C(K^{'})-1=  \frac{T_K}{3}\, +\,  \frac{T_{K^{'}}}{3} \, +\, 4-2-1\\
 &\geq &   \frac{T_K+T_{K^{'}}}{3} \, + \, 1.\\
\end{eqnarray*}}
Hence the conclusion follows.

\end{proof}

\section{Calculations of crosscap numbers}\label{calculations} 

\subsection{Lower exact bounds}  In this subsection we provide constructions and examples  of families of alternating  alternating links for  whichTheorem \ref{chi-etimate1} gives the exact value of  the crosscap number.

 Let $G$ be a trivalent planar graph and let $N(G)$ denote a neighborhood of $G$ on the plane. The boundary $\partial N(G)$ is a link. For each edge of $G$ we have two parallel arcs belonging
on different components of $\partial N(G)$. 
Construct a diagram of
a new link by adding a number of half twists between these parallel  arcs of the components $\partial N(G)$ corresponding to each edge of $G$. 
Suppose that for each edge of $G$ we add at least three crossings on $D(K)$.  We can do this so that the resulting diagram is alternating. Depending on the numbers of twists we put we may obtain a knot or a multi-component link. Let $D(K)$
denote any alternating projection obtained this way and 
let $S_G$ be a state surface obtained from Algorithm \ref{AKalgor}  applied to $D(K)$: Since each twist region contains bigons,
$S_G$ corresponds to the Kauffman state of $D(K)$ that resolves all the crossings so that the bigons are state circles.  To analyze these surfaces further we need a definition.

\begin{define} A normal disk $D$ in a polyhedron $P\in \{ P_1, P_2\}$ is called an ideal triangle if $\partial D$ intersects exactly three boundary faces
of $\partial E(L)$ and it intersects no interior edges of $P$.
\end{define}

\begin{corollary} \label{exact} Let $K$ be a  $k$-component link with a prime, twist-reduced alternating diagram $D(K)$ constructed
from a trivalent planar  graph $G$ as above.
Then we have

$$C(K)=  \left\lceil \frac{t}{3} \right\rceil\, + \, \epsilon-k =\left\lceil \frac{T_K}{3} \right\rceil\, + \, \epsilon-k.
  $$
where $\epsilon =2$ if $S_G$ is non-orientable and $\epsilon=3$ otherwise.
\end{corollary}

\begin{proof} By assumption $D(K)$ is obtained by adding twists along components of the boundary $\partial N(G)$ of  a regular neighborhood of a planar
trivalent graph.  
The surface $S_G$ is obtained by $N(G)$ by similar twisting and the
augmented link of Lemma \ref{incomplement} is obtained by adding a component encircling  each twist region of $D(K)$.
In order to calculate $\chi(S_G)$ we need some more detailed information about  the polyhedral decomposition $\{P_1, P_2\}$ of $E(L)$ \cite{purcellsurvey}.  The surface $S_G$ gives rise to surface $S'_G$ in $E(L)$; we will  calculate $\chi(S'_G)$.
Let ${\mathcal D}$ denote the union of the crossing disks  bounded by the crossing circles of $L$.
Each disk intersects the projection plane in a single arc. We may isotope the interior of $S'_G$ so that it is disjoint from the intersections of $\mathcal D$ with the projection plane.
Let $P\in \{P_1, P_2\}$. Recall that (before truncation)  all the vertices of $P$  are of valence four 
and they correspond  to crossing circles of $L$ and to arcs of $K\setminus {\mathcal D}$ and
the faces
can be colored in a checkerboard fashion (in shaded and white) as follows: 
\begin{enumerate}

\item The shaded faces of $P$  correspond to the crossing disks: Each disk  $D$ gives  two triangular 
shaded faces of $P$ meeting at an ideal vertex corresponding to the crossing circle $\partial D$  (``bowties").
\item The edges of $P$ come from the intersections of $\mathcal D$ with the projection plane.
\item The white faces of $P$ correspond to regions  of $D(K)$ on the projection plane.
\end{enumerate}

To a vertex $v\in G$ there corresponds a triangular region of $D(K)$  that is a neighborhood of $v$, around which the three twist
regions $D(K)$,  corresponding to the edges of $G$ emanating from $v$, meet.  Let ${\mathcal D}_v$ denote the union of the three crossing disks corresponding to the three twist regions of $D(K)$ around $v$.
This triangular region will become
an ideal triangular  white face, say $D_v$,   after truncating the vertices of $P$. See Figure \ref{panel}. The ideal vertices of $D_v$ come from the arcs of $K\setminus {\mathcal D}_v $  that surround $v$. Each of the three interior edges of $P$ on $\partial D_v$  is attached to a bowtie:
two shaded faces meeting at an   ideal vertex coming from one of the crossing disks 
in ${\mathcal D}_v$.

\begin{figure}[ht]
\includegraphics[scale=.18]{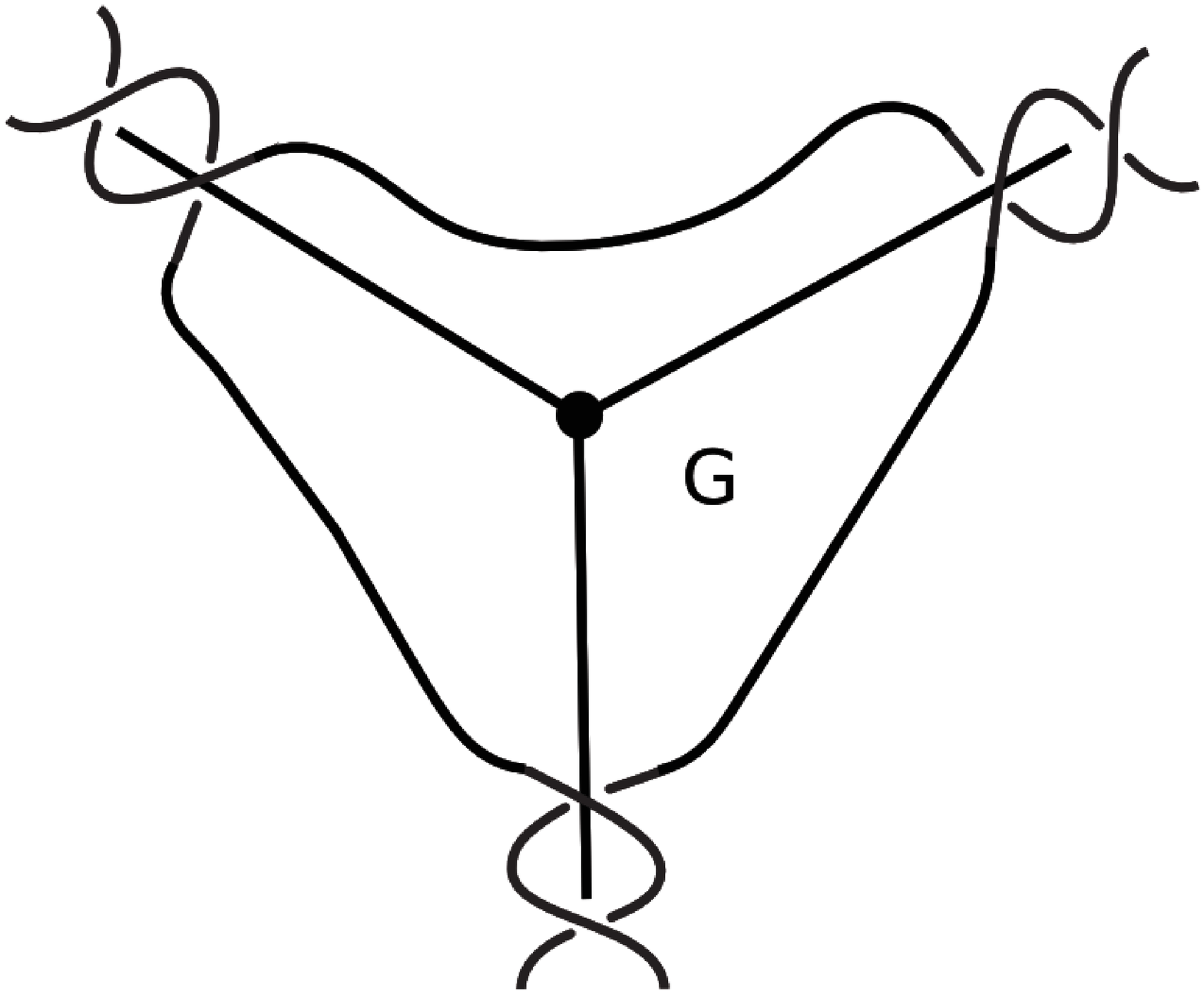}
\hspace{0.4cm}
\includegraphics[scale=.18]{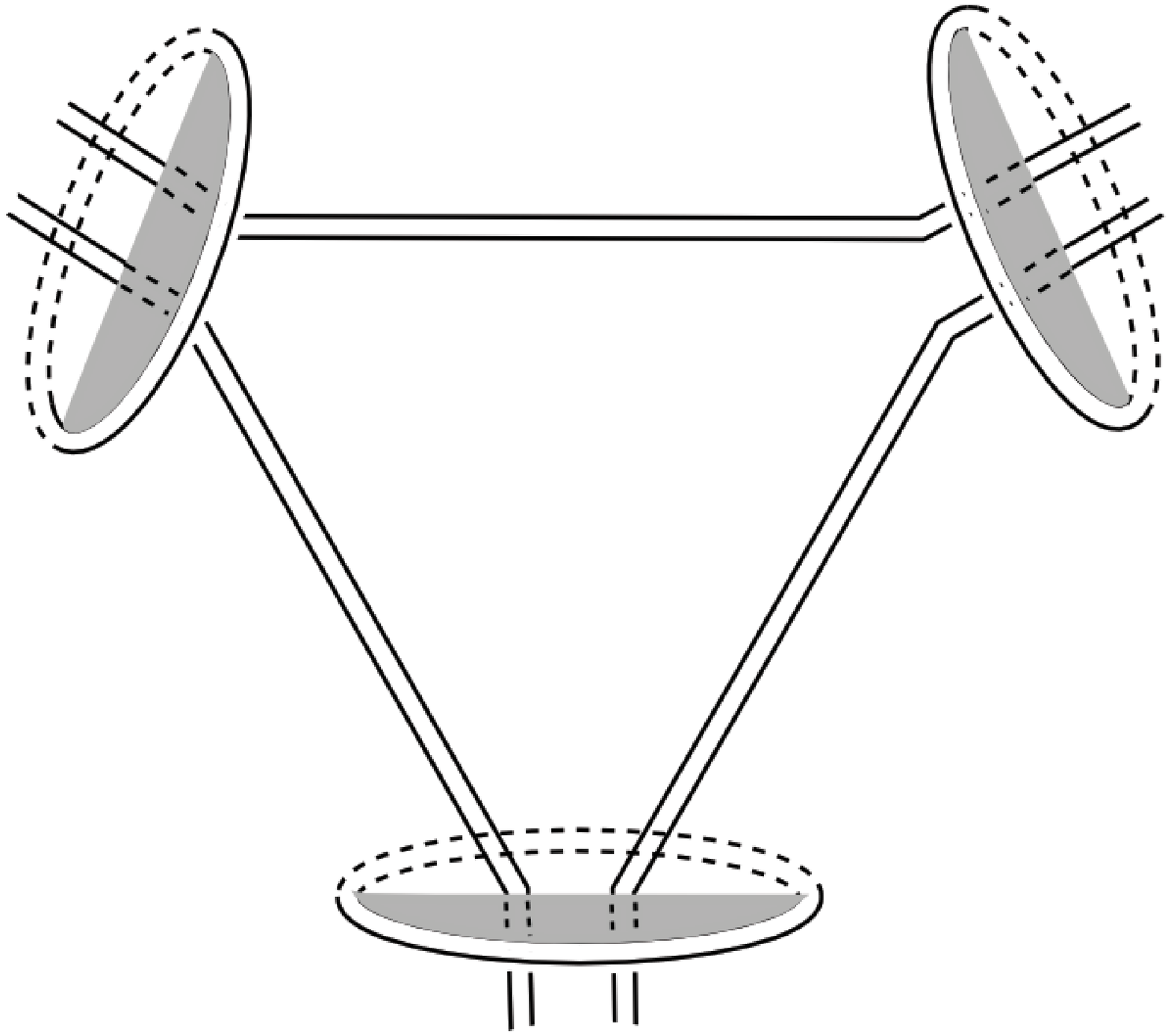}
\hspace{0.4cm}
\includegraphics[scale=.20]{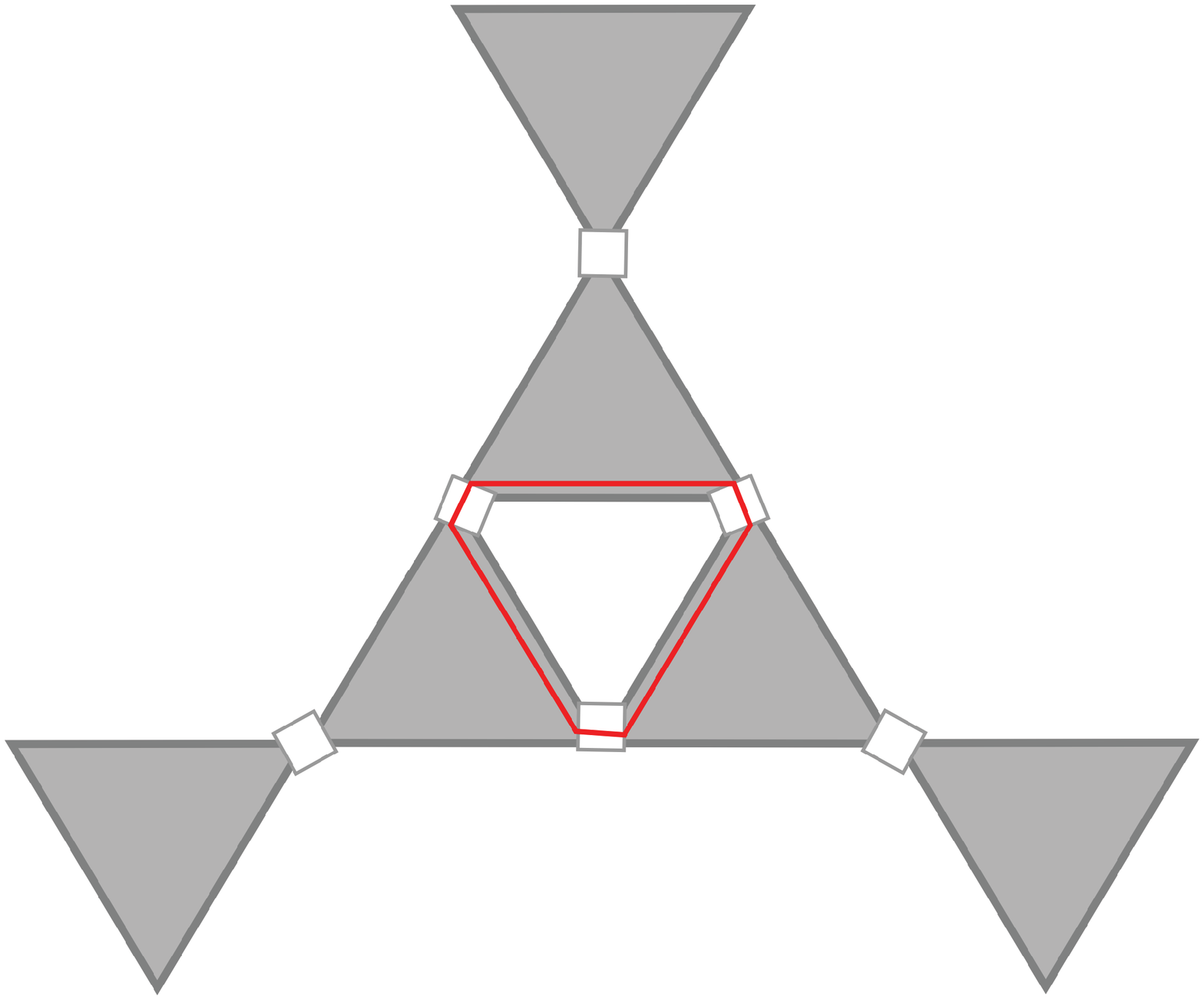}
\caption{From left to right: The portion  of $D(K)$ around a vertex of $G$,  the corresponding  potion of the augmented link
and portion of the polyhedral decomposition. An ideal triangle is indicated by the red line.}
\label{panel} 
\end{figure} 

Since the surface $S_G$ doesn't intersect the crossing circles, 
 $S'_G$  is disjoint from the ideal vertices corresponding to disks in  ${\mathcal D}_v$.
Furthermore, since we arranged so that the interior of $S'_G$ is disjoint from the intersection of ${\mathcal D}$ with the projection plane,  
$S'_G$  is disjoint from the boundary faces of $P$ corresponding to ${\mathcal D}_v$.
 It follows that after putting $S'_G$
into normal form in $P$ we will have a 
 normal triangle (an ideal triangle surrounding $D_v$) corresponding to the part of $S'_G$ around $v$.
Now $S'_G$ will consist of a collection of ideal triangles and  the only contributions to $a(S'_G)$ will come  from these  ideal triangles.
Let $\mathcal T$ be such a triangle and let $\gamma_i$, $i=1, 2, 3$,
denote the three arcs of $\partial {\mathcal T}$ on $\partial E(L)$.
By Definition \ref{area} and the definition of combinatorial lengths we have

$$a(\mathcal T)= \pi= 3\cdot \frac{\pi}{3}=\Sigma_{i=1}^3  l(\gamma_i, \mathcal T).$$
The number of ideal triangles in $S'_g$ is at most $ 2t/3$. Hence the calculation
in the proof of Theorem \ref{chi-etimate1} gives
$$a(S'_G)=  \frac{2t\pi}{3},$$
and by Theorem \ref{gaussBonnet}
we obtain $$-\chi(S_G)=-\chi(S'_G)=\left\lceil \frac{t}{3} \right\rceil.$$
Now if $S_G$ is is non-orientable then, by Theorem \ref{AK}, $C(K)=C(S_G)=-\chi(S)\, + \, 2-k$.
If $S_G$ is orientable then   
$C(K)=2g(K)+1=-\chi(S_G)+3-k$. In both cases the desired result follows, since $D(K)$ is prime and twist reduced and we have $t=T_K$.
\end{proof}

\begin{example} \label{petz}Let $D(K)=P(p_1, p_2, \dots, p_N)$ denote the standard diagram of an alternating $N$-string pretzel knot.
Suppose that, for $i=1,\ldots, N$,  $\abs{p_i}>2$. We can see that the surface  $S$ corresponding to Theorem \ref{AK}
is the pretzel surface consisting of two disks and $N$-twisted bands; one for each twist region of $D(K)$.
Augment the diagram $D(K)$  by adding  a crossing circle at each twist region so that the surface $S$ intersects
each crossing disk in a single arc only.
This surface is disjoint from the crossing circles of the fully  augmented link $L$.
Furthermore, the surface $S$ is essential in $E(K)$ and thus in $E(L)$.
Hence we may put it in normal form to calculate the area $a(S)$. By an argument similar to this in the proof of Corollary \ref{exact}, it follows that 
the contributions to $a(S)$ come from two identical, normal  $N$-gons, $D_1, D_2$
such that $\partial D_i$ intersects $N$-boundary faces of the polyhedral decomposition and it intersects no interior edges.
We have $$a(S)=a(D_i)+a(D_2)=2\pi N-4\pi=2\pi (T_K-2),$$ and thus $-\chi(S)=N-2$.
If all, but one, $p_i$ are odd the $S$ is non-orientable and thus  $C(K)=N-1=T_K-1$. 
If all the $p_i$'s are odd, then $S$ is orientable and then, $C(K)=N=T_K$. Note that the crosscap numbers of
pretzel knots have been calculated in \cite{pretzel}.
\end{example}

\subsection{Low crossing knots}The crosscap numbers of all alternating links up to 9  crossings are known.
Knotinfo \cite{knotinfo} provides an upper and a lower bound for the crosscap numbers of all knots up to 12 crossings  for which the 
exact values of crosscap numbers are not known.   Note that in all, but a handful of  cases, where the crosscap number is  not determined, the lower bound given in Knotinfo is 2. 
There are 1778 prime,  alternating knots with crossing numbers $10\leq c\leq 12$.  For these knots
we  calculated the quantity $T_K:= \abs{\beta_K} + \abs{\beta'_K}$  using the Jones polynomial value  given in Knotinfo and we compared our crosscap number lower bound with
the one given in therein. For  1472 of these knots our lower bound is better than the one given in Knotinfo
and for 283 of them our lower bound agrees with the upper bound in there;  thus we are able to calculate the exact value of the crosscap number in these cases.
These data have now been uploaded in Knotinfo by Cha and Livingston.

For example, 
in Table \ref{three}
we have all the 37 alternating knots $K$  for which Knotinfo states $2\leq C(K)\leq 3$,
 together with the value of the corresponding
quantity $T_K$. 
In all the 37 cases the lower bound above is also 3; thus for all these knots we can determine the crosscap number to be 3.

\begin{table}[]
\begin{center}
\begin{tabular}{|c|c|c|c|c|c|c|c|c|}
\hline
$K$& $T_K$ & $K$ & $T_K$ & $K$ & $T_K$ & $K$ & $T_K$  \\
\hline
$10_{85}$ &6& $10_{93}$ &6& $10_{100}$  & 6 &  $11{\textrm{a}}_{74}$ &5 \\
\hline
 $11{\textrm{a}}_{97}$ &5& $11{\textrm{a}}_{223}$ & 5 &$11{\textrm{a}}_{250}$ &5 & $11{\textrm{a}}_{259}$ &5 \\
\hline
 $11{\textrm{a}}_{263}$ & 4 &$11{\textrm{a}}_{279}$ &6 
& $11{\textrm{a}}_{293}$ &6& $11{\textrm{a}}_{313}$ & 6  \\
\hline
 $11{\textrm{a}}_{323}$ &6 & $11{\textrm{a}}_{330}$ &6& $11{\textrm{a}}_{338}$ & 4 & $11{\textrm{a}}_{346}$ & 6 \\
\hline $12{\textrm{a}}_{0636}$ &5& $12{\textrm{a}}_{0641}$ & 4 & $12{\textrm{a}}_{0753}$ &5& $12{\textrm{a}}_{0827}$ &5 \\
\hline
 $12{\textrm{a}}_{0845}$ & 5 & $12{\textrm{a}}_{0970}$ &6& $12{\textrm{a}}_{0984}$ &6 & $12{\textrm{a}}_{1017}$ & 6 \\ 
\hline
$12{\textrm{a}}_{1031}$ &5 
& $12{\textrm{a}}_{1095}$ &6  & 
$12{\textrm{a}}_{1107}$ & 6 & $12{\textrm{a}}_{1114}$ &6 \\
\hline
 $12{\textrm{a}}_{1142}$ &5 & $12{\textrm{a}}_{1171}$  &6 & $12{\textrm{a}}_{1179}$ &6 & $12{\textrm{a}}_{1205}$ &6 \\
\hline
$12{\textrm{a}}_{1220}$ & 6 & $12{\textrm{a}}_{1240}$ &6&
$12{\textrm{a}}_{1243}$ &4 & $12{\textrm{a}}_{1247}$ &6 \\
\hline
 $12{\textrm{a}}_{1285}$ & 4 &- & -& - & -&- & - \\
\hline
 \end{tabular}

\end{center}

\vskip.13in
\caption{Examples of knots where the Knotinfo upper bound agrees with our lower bound.  The crosscap number is 3.}
\label{three}
\end{table}

\section{Generalizations and questions}\label{secfive}

\subsection{Non-alternating links} 
A question arising from this work  is the question of the extent to which the Jones polynomial
(coarsely) determines 
the crosscap number outside the class of alternating links. This is an interesting question that merits further 
investigation. Our contribution towards an answer to this question, in this paper, is
to provide a generalization of Theorem \ref{thm:cupjones} 
for some class of non-alternating links. To state our result we need a definition.
\begin{define} For a link diagram $D(K)$ let $S_A$ denote the state surface corresponding to the Kauffman state where all the crossings are resolved one way
and let $S_B$ denote the state surface corresponding to the state where all crossings are resolved the opposite way.

A link $K$ is called  \emph{adequate} if it admits a link diagram $D(K)$ such that none of  $S_A, S_B$ contains a half-twisted band
with both ends  attached on the same state circle.

Adequate links form a large class that contains the alternating ones but it is much wider. See \cite{lick-thistle, fkp:filling, fkp}.
\end{define}

\begin{theorem}\label{thm:cupjonesadequate}
 Let $K$ be a  $k$-component link with crosscap number $C(K)$ and let $T_K$ be as above.
Suppose that $K$ admits a connected, twist-reduced, diagram $D(K)$
that has  $t\geq 2$ twist regions, and such that each twist region of $D(K)$ contains at least six crossings.
We have

$$ \left\lceil \frac{t}{3} \right\rceil\, + \, 2-k \, \leq C(K) \, \leq \, t+2-k.
  $$
If moreover $D(K)$ is adequate then we have

$$\left\lceil \frac{ T_K}{6}\right\rceil \, +\, 2 \, -\, k \; \leq \; C(K) \; \leq \;  3T_K\, -\, k\, -\, 1.
$$

 \end{theorem}
\begin{proof} Let $S$ be a spanning surface of $K$ that is of maximal Euler characteristic over all spanning surfaces (that is both orientedable and non-orientedable).
Then $S$ gives gives an incompressible and $\partial$-incompressible surface in $E(K)$. For, surgery of $S$ along a compression  or a $\partial$-compression disk
will produce a surface of higher Euler characteristic (compare proof of Lemma \ref{normal}). 
Let $L$ be a fully  augmented link obtained from $D(K)$. Recall that each crossing circle added  in this process bounds a disk $D$ whose interior
is pierced exactly twice by $K$. We will isotope $S$ in the complement of $K$ so that $S\cap D$ is minimized. 
In this process, since  $S$ is incompressible, if  there is a simple closed curve in $\delta \subset S\cap D$
that bounds a disk in $D$ with its interior disjoint from $S$, then we can eliminate $\delta$ by isotopy of $S$ in the complement of $L$.
Similarly we can eliminate arc components of $S\cap D$ that cut off disks on $D$ with their interior disjoint from $S$.
Finally, simple closed curves
that are parallel to $\partial D$ can be eliminated by sliding $S$ off the boundary of $D$.

The surface $S$ gives rise to a surface $F$ in $E(L)$. We claim that 
$F$ is incompressible and $\partial$-incompressible  in $E(L)$. 
To see that, suppose that there is an essential simple closed curve $\gamma \subset S$ 
that bounds a compressing disk in $\Delta\subset E(L)$. Since $S$ is incompressible, $\gamma$ must bound a disk $\Delta' \subset S$ whose interior is 
intersected by the crossing circles of $L$. Now $\Delta\cup \Delta'$ bounds a 3-ball that can be used to produce an isotopy that reduces the intersection of $S$
and the crossing disks of $L$; contradiction.

Now we argue that $F$ is $\partial$-incompressible. To that end, suppose that $F$ admits a $\partial$-compression disk
$\Delta$, with $\partial \Delta=\gamma \cup \delta$, where $\gamma$ is a spanning arc in $F$ and $\delta$ is an arc on
a component $T \subset \partial E(L)$. Assume, for a moment,  that $T$ corresponds to a component of $K$. Since $S$ is $\partial$-incompressible in $E(K)$, the arc $\gamma$ must cut a disk $\Delta'\subset S$ whose interior is pieced by the crossing circles of $L$.
Since $S$ is incompressible, the boundary of the disk $\Delta\cup \Delta'$ also bounds a disk $\Delta''\subset S$. Now $\Delta''\cup \Delta\cup \Delta'$ bounds a 3-ball that can be used to produce an isotopy that reduces the intersections of $S$ with the crossing disks of $L$. This is a contradiction.

Suppose now that $T$ is a component of $\partial E(L)$ that corresponds to a crossing circle of $L$.
Since $S$ intersects crossing circles an even number of times, $T\setminus \partial S$ has at least two components.
Thus $\delta$ lies on an annulus of $A\subset T\setminus \partial S$ and either it cuts off a disk on $A$ or it runs between different components of $A$. Now the usual argument that shows that an orientable, spanning surface of a link  that is incompressible,
has to be $\partial$-incompressible applies to obtain a contradiction (see \cite[Lemma 1.10]{hatcher}). Thus
the punctured surface $S$ is essential in $E(L)$.

Now we may replace $S$ by a surface, of the same orientability and Euler characteristic,
that  is in normal form with respect to the polyhedral
decomposition of $E(L)$ \cite[Theorem 2.8]{fg:arborescent}. Now  Corollary \ref{twisted} applies.

If $S$ is non-orientable, we have $C(S)=C(K)$ and by Corollary \ref{twisted}
we have $$C(K)\geq \left\lceil  \frac{t}{3}\right\rceil \, +\,  2\, -\, k .$$
If $S$ is orientable then $C(K)=2g(S)+1=2g(K)+1$ and by Theorem \ref{chi-etimate} again we have $C(K)\geq \left\lceil  \frac{t}{3}\right\rceil \, +\,  2\, -\, k .$
Combining these inequalities with Lemma \ref{uppert} we have
$$ \left\lceil \frac{t}{3} \right\rceil\, + \, 2-k \, \leq C(K) \, \leq \, t+2-k,
  $$
\noindent which proves the first part of the theorem. Suppose now that
$D(K)$ is also adequate. Then  \cite[Theorem 1.5]{fkp:filling} implies that

$$\left\lceil \frac{t}{3} \right\rceil\,  \, \leq T_K  \, \leq \, 2t.
  $$
Now combining the last two inequalities gives the desired result.
  \end{proof}

Theorem \ref{thm:cupjonesknots} should be compared with Murasugi's classical result \cite{murasugi} that the Alexander polynomial 
determines the Seifert genus of alternating knots. The Alexander polynomial doesn't determine the genus of non-alternating knots.
However the Knot Floer Homology (the categorification of the Alexander polynomial) determines the genus of all knots \cite{OS}.
A related question is the question of whether the  Khovanov homology \cite{Khovanov} of   knots (the categorification of the Jones polynomial)
is related to the crosscap number and the extent to which the former determines the later.  
 We ask  the following questions:
\begin{enumerate}
\item  For which links does the Jones polynomial (coarsely) determines $C(K)$?
\smallskip
\smallskip

\item Are there two sided bounds of $C(K)$ of \emph{every} link $K$ in terms of the Khovanov homology of $K$?
\end{enumerate}

\subsection{Combinatorial area and colored Jones polynomials} 
The colored Jones polynomial of  a link $K$, is a sequence of Laurent polynomial invariants.
$$J^n_K(t)= \alpha_n t^{j(n)}+ \beta_n t^{j(n)-1}+ \ldots + \beta'_n
t^{j'(n)+1}+ \alpha'_n t^{j'(n)}, \ \ \ n=1,2,\ldots$$
with $J^2_K(T)$ being the ordinary Jones polynomial.
It is known that, for every $i>0$, the absolute values of the $i$-th and the $i$-th to last coefficients of  
$J^n_K(t)$ stabilize when $i>n$ \cite{Armond, gaLe}.
For instance we have $\abs{\beta'_K}:=\abs{\beta'_n}$, and $\abs{\beta_K}:=\abs{\beta_n}$
for $n>1$.
The proofs of Theorem \ref{thm:cupjones},  Corollary \ref{exact}, as well as  Example \ref{petz}, indicate that the quantity
 $$\frac{T_K}{2\pi}=\frac{\abs{\beta'_K}+\abs{\beta_K}}{2\pi},$$
is related to combinatorial areas of  a normal surfaces in  augmented links   of $K$. 
One may ask whether the higher order stabilized coefficients of the colored Jones polynomials have similar interpretations.
We will investigated this question in a future paper.

\vskip 0.07in

\noindent {\bf Acknowledgement.}We thank Dave Futer and Jessica Purcell for several conversations and clarifications on the geometry of augmented links.
Part of the results in this paper  were obtained while Kalfagianni was visiting the Erwin Schr\"odinger International Institute for Mathematical Physics during the program
 ``Combinatorics, Geometry and  Physics" in the summer of 2014. She thanks the organizers of the program and the staff at ESI for providing excellent working conditions.

\bibliographystyle{plain} \bibliography{biblio}

\begin{thebibliography}{10}

\bibitem{Adamsstate}
Colin Adams and Thomas Kindred.
\newblock A classification of spanning surfaces for alternating links.
\newblock {\em Algebr. Geom. Topol.}, 13(5):2967--3007, 2013.

\bibitem{Adamsaug}
Colin~C. Adams.
\newblock Augmented alternating link complements are hyperbolic.
\newblock In {\em Low-dimensional topology and {K}leinian groups,
  ({C}oventry/{D}urham, 1984)}, volume 112 of {\em London Math. Soc. Lecture
  Note Ser.}, pages 115--130. Cambridge Univ. Press, Cambridge, 1986.

\bibitem{genus}
Ian Agol, Joel Hass, and William Thurston.
\newblock The computational complexity of knot genus and spanning area.
\newblock {\em Trans. Amer. Math. Soc.}, 358(9):3821--3850, 2006.

\bibitem{Armond}
Cody Armond.
\newblock The head and tail conjecture for alternating knots.
\newblock {\em Algebr. Geom. Topol.}, 13(5):2809--2826, 2013.

\bibitem{batson}
Joshua Batson.
\newblock Nonorientable four-ball genus can be arbitrarily large.
\newblock arXiv:1204.1985.

\bibitem{burton}
Benjamin~A. Burton and Melih Ozlen.
\newblock Computing the crosscap number of a knot using integer programming and
  normal surfaces.
\newblock {\em ACM Trans. Math. Software}, 39(1):Art. 4, 18, 2012.

\bibitem{knotinfo}
Jae~Choon Cha and Charles Livingston.
\newblock Knotinfo: Table of knot invariants.
\newblock {\tt http://\allowbreak www.indiana.edu/\allowbreak \~{ }knotinfo},
  June 14 2014.

\bibitem{clark}
Bradd~Evans Clark.
\newblock Crosscaps and knots.
\newblock {\em Internat. J. Math. Math. Sci.}, 1(1):0161--1712, 1978.

\bibitem{crowell}
Richard Crowell.
\newblock Genus of alternating link types.
\newblock {\em Ann. of Math. (2)}, 69:258--275, 1959.

\bibitem{dfkls:determinant}
Oliver~T. Dasbach, David Futer, Efstratia Kalfagianni, Xiao-Song Lin, and
  Neal~W. Stoltzfus.
\newblock Alternating sum formulae for the determinant and other link
  invariants.
\newblock {\em J. Knot Theory Ramifications}, 19(6):765--782, 2010.

\bibitem{dasbach-lin:volumish}
Oliver~T. Dasbach and Xiao-Song Lin.
\newblock {A volume-ish theorem for the Jones polynomial of alternating knots}.
\newblock {\em Pacific J. Math.}, 231(2):279--291, 2007.

\bibitem{fg:arborescent}
David Futer and Fran{\c{c}}ois Gu{\'e}ritaud.
\newblock Angled decompositions of arborescent link complements.
\newblock {\em Proc. Lond. Math. Soc. (3)}, 98(2):325--364, 2009.

\bibitem{fkp:filling}
David Futer, Efstratia Kalfagianni, and Jessica~S. Purcell.
\newblock {Dehn filling, volume, and the Jones polynomial}.
\newblock {\em J. Differential Geom.}, 78(3):429--464, 2008.

\bibitem{fkp:farey}
David Futer, Efstratia Kalfagianni, and Jessica~S. Purcell.
\newblock {Cusp areas of {F}arey manifolds and applications to knot theory}.
\newblock {\em Int. Math. Res. Not. IMRN}, 2010(23):4434--4497, 2010.

\bibitem{fkp}
David Futer, Efstratia Kalfagianni, and Jessica~S. Purcell.
\newblock {\em Guts of surfaces and the colored {J}ones polynomial}, volume
  2069 of {\em Lecture Notes in Mathematics}.
\newblock Springer, Heidelberg, 2013.

\bibitem{fkp:survey}
David Futer, Efstratia Kalfagianni, and Jessica~S. Purcell.
\newblock Jones polynomials, volume, and essential knot surfaces: a survey.
\newblock In {\em Proceedings of Knots in Poland III}, volume 100, pages
  51--77. Banach Center Publications, 2014.

\bibitem{fkp:qsf}
David Futer, Efstratia Kalfagianni, and Jessica~S. Purcell.
\newblock Quasifuchsian state surfaces.
\newblock {\em Trans. Amer. Math. Soc.}, 366(8):4323--4343, 2014.

\bibitem{futer-purcell}
David Futer and Jessica~S. Purcell.
\newblock Links with no exceptional surgeries.
\newblock {\em Comment. Math. Helv.}, 82(3):629--664, 2007.

\bibitem{garoufalidis:jones-slopes}
Stavros Garoufalidis.
\newblock The {J}ones slopes of a knot.
\newblock {\em Quantum Topol.}, 2(1):43--69, 2011.

\bibitem{gaLe}
Stavros Garoufalidis and Thang T.~Q. L{\^e}.
\newblock Nahm sums, stability and the colored jones polynomial.
\newblock {\em Research in Mathematical Sciences, to appear}.

\bibitem{haken}
Wolfgang Haken.
\newblock Theorie der {N}ormalfl\"achen.
\newblock {\em Acta Math.}, 105:245--375, 1961.

\bibitem{genusalg}
Joel Hass, Jeffrey~C. Lagarias, and Nicholas Pippenger.
\newblock The computational complexity of knot and link problems.
\newblock {\em J. ACM}, 46(2):185--211, 1999.

\bibitem{hatcher}
Allen Hatcher.
\newblock Notes on basic 3-manifold topology.
\newblock {\tt http://www.math.cornell.edu/\allowbreak \~{
  }hatcher/3M/\allowbreak 3Mdownloads.html}.

\bibitem{2-bridge}
Mikami Hirasawa and Masakazu Teragaito.
\newblock Crosscap numbers of 2-bridge knots.
\newblock {\em Topology}, 45(3):513--530, 2006.

\bibitem{pretzel}
Kazuhiro Ichihara and Shigeru Mizushima.
\newblock Crosscap numbers of pretzel knots.
\newblock {\em Topology Appl.}, 157(1):193--201, 2010.

\bibitem{Jaco-Rub}
William Jaco and J.~Hyam Rubinstein.
\newblock P{L} equivariant surgery and invariant decompositions of
  {$3$}-manifolds.
\newblock {\em Adv. in Math.}, 73(2):149--191, 1989.

\bibitem{Kaufjones}
Louis~H. Kauffman.
\newblock State models and the {J}ones polynomial.
\newblock {\em Topology}, 26(3):395--407, 1987.

\bibitem{Khovanov}
Mikhail Khovanov.
\newblock A categorification of the {J}ones polynomial.
\newblock {\em Duke Mathematical Journal}, 101(3):359--426, 2000.

\bibitem{lack-surg}
Marc Lackenby.
\newblock Word hyperbolic dehn surgery.
\newblock {\em Invent. Math.}, 140(2):243--282, 2000.

\bibitem{lackenby:volume-alt}
Marc Lackenby.
\newblock The volume of hyperbolic alternating link complements.
\newblock {\em Proc. London Math. Soc. (3)}, 88(1):204--224, 2004.
\newblock With an appendix by Ian Agol and Dylan Thurston.

\bibitem{lickorish:book}
W.~B.~Raymond Lickorish.
\newblock {\em An introduction to knot theory}, volume 175 of {\em Graduate
  Texts in Mathematics}.
\newblock Springer-Verlag, New York, 1997.

\bibitem{lick-thistle}
W.~B.~Raymond Lickorish and Morwen~B. Thistlethwaite.
\newblock Some links with nontrivial polynomials and their crossing-numbers.
\newblock {\em Comment. Math. Helv.}, 63(4):527--539, 1988.

\bibitem{menasco:incompress}
William~W. Menasco.
\newblock Closed incompressible surfaces in alternating knot and link
  complements.
\newblock {\em Topology}, 23(1):37--44, 1984.

\bibitem{murakamisurvey}
Hitoshi Murakami.
\newblock An introduction to the volume conjecture.
\newblock In {\em Interactions between hyperbolic geometry, quantum topology
  and number theory}, volume 541 of {\em Contemp. Math.}, pages 1--40. Amer.
  Math. Soc., Providence, RI, 2011.

\bibitem{upper}
Hitoshi Murakami and Akira Yasuhara.
\newblock Crosscap number of a knot.
\newblock {\em Pacific J. Math.}, 171(1):261--273, 1995.

\bibitem{murasugi}
Kunio Murasugi.
\newblock On the {A}lexander polynomial of the alternating knot.
\newblock {\em Osaka Math. J.}, 10:181--189; errata, 11 (1959), 95, 1958.

\bibitem{murasugitait}
Kunio Murasugi.
\newblock Jones polynomials and classical conjectures in knot theory.
\newblock {\em Topology}, 26(2):187--194, 1987.

\bibitem{OS}
Peter Ozsv{\'a}th and Zolt{\'a}n Szab{\'o}.
\newblock Holomorphic disks and genus bounds.
\newblock {\em Geom. Topol.}, 8:311--334, 2004.

\bibitem{purcellsurvey}
Jessica~S. Purcell.
\newblock An introduction to fully augmented links.
\newblock In {\em Interactions between hyperbolic geometry, quantum topology
  and number theory}, volume 541 of {\em Contemp. Math.}, pages 205--220. Amer.
  Math. Soc., Providence, RI, 2011.

\bibitem{torus}
Masakazu Teragaito.
\newblock Crosscap numbers of torus knots.
\newblock {\em Topology Appl.}, 138(1-3):219--238, 2004.

\end{thebibliography}
\end{document}